\documentclass[reqno,twoside,12pt]{amsart}

\usepackage{amsmath,amsfonts,calrsfs,amssymb,color,fullpage}

\newtheorem{Theorem}{Theorem}[section]
\newtheorem{Definition}[Theorem]{Definition}
\newtheorem{Proposition}[Theorem]{Proposition}
\newtheorem{Lemma}[Theorem]{Lemma}
\newtheorem{Corollary}[Theorem]{Corollary}
\newtheorem{Remark}[Theorem]{Remark}
\newtheorem{Example}[Theorem]{Example}
\newtheorem{Hypothesis}[Theorem]{Hypothesis}
\makeatletter
\@addtoreset{equation}{section}

\makeatother

\def\oo{\mathaccent23}

\def\R{\mathbb R}
\def\N{\mathbb N}

\def\O{\mathcal O}
\def\eps{\varepsilon}
\def\ds{\displaystyle}

\newcommand{\one}{1\!\!\!\;\mathrm{l}}
\newcommand{\Tr}{\operatorname{Tr}}
\newcommand{\di}{\operatorname{div}}
\newcommand{\Det}{\operatorname{Det}}
\newcommand{\sign}{\operatorname{sign}}

\title[Traces]{Traces of Sobolev functions on regular surfaces in infinite dimensions}

\author[P. Celada, A. Lunardi]{Pietro Celada, Alessandra Lunardi}
\address{
Dipartimento di Matematica\\
Universit\`a di Parma\\
Parco Area delle Scienze, 53/A\\
43124 Parma, Italy}
\email{alessandra.lunardi@unipr.it}
\email{pietro.celada@unipr.it}



\begin{document}

\begin{abstract}  
In a Banach space $X$ endowed with a nondegenerate Gaussian measure,  we consider Sobolev spaces of real functions defined in a sublevel set $\O= \{x\in X:\;G(x) <0\}$ of a Sobolev nondegenerate function $G:X\mapsto \R$. We define the traces at $G^{-1}(0)$ of the elements of $W^{1,p}(\O, \mu)$ for $p>1$, as elements of  $L^1(G^{-1}(0), \rho)$ where $\rho$ is the surface measure of Feyel and de La Pradelle. The  range of the trace operator is contained in $L^q(G^{-1}(0), \rho)$ for $1\leq q<p$ and even in $L^p(G^{-1}(0), \rho)$ under further assumptions. 
If $\O$ is a suitable halfspace, the range is characterized as a sort of fractional Sobolev space at the boundary. 

An important consequence of the general theory is an integration by parts formula for Sobolev functions, which involves their traces at $G^{-1}(0)$.  
\end{abstract}

 \maketitle

\section{Introduction}
Let $X$ be a separable Banach space with norm $\|\cdot\|$,  endowed with a nondegenerate centered Gaussian measure $\mu$. The relevant Cameron--Martin space is denoted by $H$, its scalar product by $\langle \cdot , \cdot \rangle_H$ and its norm by $|\cdot |_H$. The covariance operator is denoted by $Q:X^*\mapsto X$, where $X^*$ is the dual space of $X$. 
 
We consider subsets of $X$ of the type $\O = \{x\in X: \;G(x) <0\}$ for suitable $G:X\mapsto \R$,  and Sobolev spaces of real valued functions defined in $\O$. 

The aim of this paper is to set the bases of the theory of traces  of  Sobolev functions at the level sets $\{x\in X:\;G(x) = 0\}$.  Such traces belong to $L^p$ spaces with respect to a ``natural"  measure on $G^{-1}(0)$, namely  the Hausdorff--Gauss surface measure $\rho$ of Feyel and de La Pradelle  \cite{FP}, which is a generalization of the Airault-Malliavin surface measure \cite{AM}. The latter is defined only on level sets of very smooth functions $G$, while in the applications to e.g. differential equations in infinite dimensional spaces the level sets are not usually so smooth. In this context smoothness is intended in terms of Sobolev regularity, not in terms of $C^k$ regularity. Precisely, we consider a (suitable version) of   $G$ satisfying $\mu(G^{-1}(-\infty, 0 )) >0$ and $G^{-1}(0)\neq \emptyset $ to avoid trivialities and pathological cases, and such that 

\begin{enumerate}
\item $G \in W^{2,q}(X , \mu)$ for each  $q>1$, 
\item there exists $  \delta >0$ such that $1/|D_HG |_H \in L^{q}(G^{-1}((-\delta, \delta), \mu))$ for each $q>1$. 
\end{enumerate}

The spaces $W^{2,q}(X , \mu)$ considered here are the usual Sobolev spaces of the Malliavin Calculus, and $D_HG $ denotes the generalized gradient of $G$ along $H$, see sect. \ref{Notation and preliminaries}.  For $\O $ and $G^{-1}(0)$ to be well defined we have to fix a version of $G$. If $X=\R^n$, $G$ has a version which belongs to $C^{1+\alpha}_{loc}(\R^n)$ for every $\alpha\in (0,1)$. For such a version, $\O$ is an open set and (2) implies that the gradient of $G$ does not vanish at $G^{-1}(0)$ (if $\O$ is bounded, (2) is in fact equivalent to the fact  that the gradient of $G$ does not vanish at $G^{-1}(0)$).  So, $G^{-1}(0)= \partial \O$ is $C^1$ hypersurface  that is locally $C^{1+\alpha}$ for every $\alpha\in (0,1)$.
In infinite dimensions there are no Sobolev embeddings, and $G$ may fail to have a continuous version. We fix once and for all a Borel version of $G$ which is  $C_{2,q}$-quasicontinuous for every $q$ (see sect. \ref{Notation and preliminaries} for precise definitions and references), that we still call $G$. Still, (1) and (2) may be seen as mild  regularity assumptions on $G$ and  on the set $G^{-1}(0)$. In particular, (2) is a mild non degeneracy assumption, and in fact functions satisfying (2) are sometimes called ``nondegenerate". Since the measure $\rho$ does not charge sets with vanishing $C_{1,q}$ capacities, the results are independent of the choice of the quasicontinuous version. However, on a first reading one may skip technicalities about Gaussian capacities and assume that $G$ is smooth. Indeed, the results are still meaningful for smooth $G$ and 
many difficulties and open problems are independent of the regularity of $G$. 

The Sobolev spaces $W^{1,p}(\O, \mu)$ are defined as the closure of the set of the Lipschitz continuos functions in the Sobolev norm. 
More precisely, we prove that the operator $Lip(\O)\mapsto   L^p(\O, \mu; H)$, $\varphi \mapsto  (D_H\widetilde{\varphi})_{|\O}$, is well defined and closable, and we denote by 
$W^{1,p}(\O, \mu)$ the domain of its closure (still denoted by $D_H$).  Here $\widetilde{\varphi}$ is any Lipschitz extension of $\varphi $ to the whole $X$. 

If $\O =X$, or if $X$ is finite dimensional,  there are other well known equivalent definitions of Sobolev spaces (e.g. \cite[Ch. 5]{Boga}), for instance through weak derivatives and through the powers of the standard Ornstein--Uhlenbeck operator.  In contrast, in the infinite dimensional case if $\O \neq X$ the equivalence of different reasonable definitions is not obvious. Our  choice is motivated by our approach to the trace theory. 

The starting point, and main tool for the definition of traces, is the integration formula
\begin{equation}
\label{parti}
\int_{\O} D_k\varphi \,d\mu =   \int_{\O} \hat{v}_k\varphi \,d\mu + \int_{G^{-1}(0)} \frac{D_k G}{|D_HG|_H }\varphi \,d\rho , \quad k\in \N ,
\end{equation}
that holds for every Lipschitz function $\varphi :X \mapsto \R$. 
Here $D_k\varphi= \langle D_H\varphi, v_k\rangle _H$ denotes the generalized derivative in the direction of $v_k$, where $\{v_k :\;k \in \N\}$ is an orthonormal basis of the Cameron--Martin space, and $ \hat{v}_k  $ is the element of $L^2(X, \mu)$ associated to $v_k$, namely the
unique  $g $ in the $L^2(X, \mu)$-closure of the dual space $X^*$   such that $x'(v_k) = \int_X x'(x)g(x)\,d\mu$ for every $x'\in X^*$.  If $\O = X$, \eqref{parti} without the surface integral is the usual integration formula for Gaussian measures (e.g., \cite[Ch. 5]{Boga}). If $\O\neq X$, the vector $D_HG/|D_HG|_H$ in the surface integral  plays the role of the unit exterior normal vector in the surface integral.

To prove \eqref{parti} we need to rework the Feyel's proof of the continuity of the densities of suitable measures, which is the subject of section \ref{sect:cont}, and constitutes the technical part of the paper.  

With the aid of \eqref{parti} and of its variants we follow the procedure of \cite{DPL} to show that  for every $1\leq q<p$ there exists $C_{p,q}$ such that 
\begin{equation}
\label{maggtraccia}
\int_{G^{-1}(0)} | \varphi|^q   \,d\rho  \leq C_{p,q} \|\varphi\|_{ W^{1,p}(\O, \mu)}^q ,
\end{equation}
for every Lipschitz continuous function $\varphi$. This allows to define the trace $\Tr \varphi$ at $G^{-1}(0)$ of every $\varphi \in W^{1,p}(\O, \mu)$ as an element of $L^q(G^{-1}(0),  \rho)$, for every $q\in [1, p)$, just  approximating by Lipschitz continuous functions. In this way, \eqref{parti}  and \eqref{maggtraccia} are satisfied by every $\varphi \in W^{1,p}(\O, \mu)$, with $\Tr \varphi$ replacing $\varphi$ in the surface integrals. Note that the integral $\int_{\O} \hat{v}_k\varphi\, d\mu$ is meaningful for every $\varphi\in  W^{1,p}(\O, \mu)$ since $ \hat{v}_k\in L^q(X, \mu)$ for every $q>1$. 

Under further hypotheses on the function $G$ the trace operator is bounded from $W^{1,p}(\O, \mu)$ to $L^p(G^{-1}(0),  \rho)$, for every $p>1$. But in several important examples such hypotheses are not fulfilled, and we can only prove that the trace operator is bounded from $W^{1,p}(\O, \mu)$ to $L^q(G^{-1}(0),  \rho)$ for $1\leq q<p$. This phenomenon is not related to the smoothness of $G$, for instance if $X$ is a Hilbert space and $G(x) = \|x\|^2-1$ then $\O$ is the unit open ball and $G$ is smooth, however we do not know whether the trace operator maps $W^{1,p}(B(0,1), \mu)$ into $L^p(\partial B(0,1),  \rho)$. A detailed discussion is in section \ref{Examples}, where this problem is reduced to the validity of a weak Hardy type inequality. 

Even when the trace maps $W^{1,p}(\O, \mu)$ to $L^p(G^{-1}(0),  \rho)$ continuously -- for instance, in the case of halfspaces or more generally of regions below graphs of good functions -- the characterization of the range of the trace operator is not an easy problem.  
In section \ref{Examples} we characterize the range of the trace operator when $\O$ is a suitable halfspace, namely when $G(x) =   \hat{h}(x) $  for some $\hat{h}\in X^*$. Then $h := Q(\hat{h})\in H$, $X$ is splitted as the direct sum of the one dimensional subspace spanned by $h$ and a complementary subspace $Y= (I-\Pi_h)(X)$, where
$\Pi_h(x) = \hat{h}(x)h$. This  decomposition   induces the decomposition   $\mu = \mu_1 \otimes \mu_Y$, where $\mu_1 $ is the standard Gaussian measure $N_{0,1}$ in $\R$, identified with the linear  span of $h$, and  $\mu_Y= \mu \circ (I-\Pi )^{-1}$  is a centered nondegenerate Gaussian measure in $Y$. 
After the identification of  $G^{-1}(0)= \{0\} \times Y$ with $Y$, we have $ \rho= \mu_Y$, and we prove that the  space of the traces at $G^{-1}(0)$ of the elements of $W^{1,p}(\O, \mu)$ coincides with the real interpolation space $(L^p(Y,  \mu_Y), W^{1,p}(Y,  \mu_Y))_{1-1/p, p}$. The latter may be characterized in several ways, using the realization of the standard Ornstein--Uhlenbeck operator $L_Y$ in $L^p(Y,  \mu_Y)$ and of the Ornstein--Uhlenbeck semigroup $T_Y(t)$. 
In particular, for $p=2$ the space $(L^2(Y,  \mu_Y), W^{1,2}(Y,  \mu_Y))_{1/2, 2}$ is precisely the domain of the operator $(I -L_Y)^{1/4}$ in $L^2(Y,  \mu_Y)$. 

In the case $X=\R^n$, $Y= \R^{n-1}$ it is well known that the interpolation space  $(L^p(Y,dy), $ $W^{1,p}(Y, dy))_{1-1/p, p}$ coincides with the fractional Sobolev space $W^{1-1/p, p}(Y, dy)$ if $dy$ is the Lebesgue measure. If   the Lebesgue measure is replaced by a nondegenerate Gaussian measure, a similar characterization does not hold. Of course there are inclusions:  we have
$(L^p(Y,  \mu_Y), W^{1,p}(Y,  \mu_Y))_{1-1/p, p}$ $\subset$ $W^{1-1/p, p}_{loc}(Y, dy)$, but to our knowledge for halfspaces the best global result in finite dimensions is the same as in infinite dimensions, namely we do not know any characterization of $(L^p(Y,  \mu_Y), W^{1,p}(Y,  \mu_Y))_{1-1/p, p}$ as a space defined in terms of  integrals over $Y\times Y$. 

However, we remark that comparison with the finite dimensional case may be misleading. As we already mentioned, under our assumptions if $X=\R^n$  then $G^{-1}(0)= \partial \O$ is a $C^1$ hypersurface. If it is compact, or more generally if it is uniformly $C^1$, then the trace operator is bounded from $ W^{1,p}(\O, dy) $ to $ W^{1-1/p,p}(\partial \O, dH_{n-1})$, where $dH_{n-1}$ is the usual Hausdorff surface measure. In infinite dimensions, even for very smooth functions $G$ the set $G^{-1}(0)$ is not compact, and  no reasonable extension of the notion of uniformly $C^1$ boundary seems to be appropriate in our context. 

A useful result of the paper is the extension of formula \eqref{parti} to Sobolev functions $\varphi\in W^{1,p}(\O, \mu)$, where ``$\varphi$" in the boundary integral is meant as the trace of $\varphi$. It is readily extended to vector fields $\Phi\in W^{1,p}(X, \mu; H)$, obtaining the familiar formula
\begin{equation}
\label{campi}
\int_{\O} \di\;\Phi\, d\mu = \int_{G^{-1}(0)}\langle \Tr  \Phi ,  \frac{ D_HG}{|D_HG|_H}\rangle_H \,d\rho ,
\end{equation}
where $\di$ is the Gaussian divergence.  

Traces of Sobolev functions at the boundaries of very smooth sets were already considered in the papers \cite{BDPT1,BDPT2,DPL,DPL2,AVN} in connection with differential equations  in regular subsets of Hilbert spaces, with homogeneous Neumann or Dirichlet boundary conditions. In such papers the interest was mainly focused on   null traces at the boundary. 

None of the results of this paper is ``surprising", to use a fashionable word. Indeed, what is made here is to extend to the infinite dimensional case well known results from the finite dimensional setting, with the due modifications. 
What is surprising is the amount of difficulties and false friends that we encountered in this extension, and this is why we gave complete details. 
The  elementary  questions that remain open for the moment show that, as far as Sobolev spaces are concerned,  the jump between the  finite and the infinite dimensional setting is considerably big.

\section{Notation and preliminaries.}
\label{Notation and preliminaries}
 
We denote by $Q$ the covariance of $\mu$ and we fix once and for all an orthonormal basis ${\mathcal V} =\{v_k:\;k\in \N\}$ of $H$.  We recall that if $X$ is a Hilbert space and $X^*$ is canonically identified with $X$, then $Q$ is a compact self--adjoint operator with finite trace and we can choose a basis $\{e_k:\;k\in \N\}$ of $X$  consisting of  eigenvectors of $Q$, $Qe_k= \lambda_ke_k$.  The space $H$ is just $Q^{1/2}(X)$ with the scalar product $\langle h_1, h_2\rangle_H= \langle Q^{-1/2}h_1,  Q^{-1/2}h_2\rangle_X$, and the set $\{v_k:= \sqrt{\lambda_k}e_k:\;k\in \N\}$ is an  orthonormal basis of $H$.

We say that a function $f:X\mapsto \R$ is $H$-differentiable at $x$ if there is $v\in H$ such that $f(x+h)-f(x) = \langle v, h\rangle_H + o(\|h\|_H)$, for every $h\in H$. In this case $v$ is unique, and we set $D_Hf(x) := v$. Moreover for every $k\in \N$ the directional derivative
$D_kf(x) : = \lim_{t\to 0}(f(x+tv_k) - f(x))/t$  exists and coincides with $\langle D_Hf(x), v_k\rangle_H$.

It is easy to see that if $f $ is Fr\'echet differentiable at $x$ (as a function from $X$ to $\R$), then it is $H$-differentiable. 
If $X$ is a Hilbert space and $f $ is Fr\'echet differentiable at $x$, then 
$D_H f(x) = QDf(x)$, where $Df(x)$ is the usual gradient. 
 
We consider the Gaussian Sobolev spaces $W^{k,p}(X, \mu)$, $k=1, 2$, $p\geq 1$. See e.g.  \cite[Sect. 5.2]{Boga}.  $W^{1,p}(X, \mu)$ and $W^{2,p}(X, \mu)$ are the completions of the smooth cylindrical functions (the functions of the type $f(x) = \varphi(l_1(x), \ldots, l_n(x))$, for some $ \varphi\in C^{\infty}_b(\R^n)$, $l_1, \ldots, l_n\in X^*$, $n\in \N$) in the norms 
$$\begin{array}{l}
\ds \|f\|_{W^{1,p}(X, \mu)} := \|f\|_{L^p(X, \mu)} + \bigg( \int_X \bigg(\sum_{k=1}^{\infty}  (D_{k}f(x))^2\bigg)^{p/2}\mu(dx)\bigg)^{1/p} 
\\
\\
\ds =  \|f\|_{L^p(X, \mu)} + \bigg( \int_X |D_Hf(x)|_H^{p}\mu(dx)\bigg)^{1/p}, \end{array}$$
$$\|f\|_{W^{2,p}(X, \mu)} := \|f\|_{W^{1,p}(X, \mu)} + \bigg( \int_X \bigg(\sum_{h, k=1}^{\infty}   (D_{hk}f(x))^2\bigg)^{p/2}\mu(dx)\bigg)^{1/p}. $$
Such spaces are in fact identified with subspaces of $L^p(X, \mu)$ and the (generalized) derivatives along $H$, $D_Hf$ and $D^2_Hf$ are well defined and belong to 
$L^p(X, \mu; H)$, $L^p(X, \mu;  {\mathcal H}_2)$, where ${\mathcal H}_2$ is the set of all Hilbert-Schmidt bilinear forms in $H$. The (generalized) directional derivatives of $f$ along any $v_k$ are defined by $D_kf(x) = \langle D_Hf(x), v_k\rangle_H$.

We shall use the integration  formula for  
$\varphi \in W^{1,p}(X, \mu)$, $p>1$:
\begin{equation}
\label{partiX} 
\int_{X} D_k\varphi \,d\mu =   \int_{X}\hat{v}_k\varphi \,d\mu , \quad k\in \N ,
\end{equation}
where $ \hat{v}_k \in L^2(X, \mu)$ is the element of $L^2(X, \mu)$ associated to $v_k$, namely the
unique  $g $ in the closure of $X^*$ in $L^2(X, \mu)$ such that $x'(v_k) = \int_X x'(x)g(x)\,d\mu$ for every $x'\in X^*$. 
If $v_k\in Q(X^*)$ then $ \hat{v}_k \in X^*$. In any case, $ \hat{v}_k \in L^q(X, \mu)$ for every $q>1$, so that the right hand side of \eqref{partiX}  makes sense. 
 If $X$ is a Hilbert space and the basis is chosen as above, $\hat{v}_k(x) =  \langle x, v_k\rangle_X /\lambda_k$ $= \langle x, e_k\rangle_X /\sqrt{\lambda_k}$ ($=\langle x, v_k\rangle_H $ for $x\in H$). 

\vspace{3mm}

The spaces $W^{1,p}(X, \mu; H)$ that we shall consider at the end of Sect. \ref{sect:traces} are defined similarly, replacing real valued smooth cylindrical functions by the space of  $H$-valued smooth cylindrical functions, namely  the linear span of the functions such as $x\mapsto f(x)h$, where $f$ is any real valued smooth cylindrical function and $h\in H$. 
For every measurable mapping $\Phi :X\mapsto X$ and for every smooth cylindrical $f$ we define
$$\partial _{\Phi} f(x) := \lim_{t\to 0} \frac{f(x+t\Phi (x)) - f(x)}{t} , $$
whenever such limit exists. If the limit exists for a.e. $x\in X$ and a function $\beta \in L^1(X, \mu)$ satisfies
$$ \int_X \partial _{\Phi} f(x) \mu(dx) = - \int_X f(x) \beta(x) \mu(dx) $$
for every smooth cylindrical $f$, $\beta$ is called {\em Gaussian divergence} of $v$ and denoted by $\di \;\Phi$. The Gaussian divergence is a linear bounded operator from $W^{1,p}(X, \mu; H)$ to $L^p(X, \mu)$, for every $p>1$. If $\Phi (x) = \sum_{k=1}^{\infty} \varphi_k(x) v_k$, then
$$\di\,\Phi = \sum_{k=1}^{\infty} (D_k\varphi_k - \varphi_k \hat{v}_k), $$
where the series converges in $L^p(X, \mu)$. Moreover, for every vector field $\Phi\in W^{1,p}(X, \mu; H)$ and $f\in W^{1,p'}(X, \mu)$ we have
$$\int_X \langle D_Hf, \Phi\rangle_H \mu(dx) = - \int_X f\, div\,\Phi\, \mu(dx). $$
See \cite[\S 5.8]{Boga}. 

\vspace{3mm}

Let us come back to real valued  functions. There are several characterizations of the Sobolev spaces, that will be used in the sequel. One of them is through the weak derivatives. Given $f\in L^p(X, \mu)$ and $h\in H$, a function $g\in  L^1(X, \mu)$ is called weak derivative of $f$ along $h$ if for every smooth cylindrical function $\varphi$ we have
$$\int_X \partial _h \varphi(x) f(x)\, \mu(dx) = -\int_X  \varphi(x) g(x) \,\mu(dx) + \int_X \varphi(x) f(x) \hat{h}(x)\,\mu(dx).$$
For $p>1$ the space $W^{1,p}(X, \mu)$ coincides with the set of all $f\in L^p(X, \mu)$ for which there exists a mapping $\Psi \in L^p(X, \mu; H)$ such that for every $h\in H$ the function $\langle \Psi(\cdot), h\rangle_H$ is the weak derivative of $f$ along $H$. In this case, we have $\Psi = D_Hf$  (\cite[\S 5.3, Cor. 5.4.7]{Boga}).

The Sobolev spaces may be characterized also through the powers of the realization of the Ornstein--Uhlenbeck operator in $L^p(X, \mu)$. More precisely, for $p>1$ the space $W^{1,p}(X, \mu)$ is the range of $(I-L_p)^{-1/2}$ and the space
$W^{2,p}(X, \mu)$ is the range of $(I-L_p)^{-1}$, where $L_p$ is the infinitesimal generator of the Ornstein--Uhlenbeck semigroup 
$$T(t)f(x) := \int_X f(e^{-t}x + (1-e^{-2t})^{1/2}y)\mu(dy), \quad t>0,  $$
in $L^p(X, \mu)$. Accordingly, for $k=1, 2$, the $C_{k,p}-$capacity of an open set $A\subset X$ is defined by 
$$\begin{array}{lll}
C_{k,p}(A) & :=  & \inf\{ \|f\|_{L^p(X, \mu)} : \; (I-L_p)^{-k/2}f \geq 1\;\mu-a.e.\;\mbox{\rm in} \;A\}
\\
\\
& = & \inf\{ \|(I-L_p)^{k/2}g\|_{L^p(X, \mu)} : \; g\geq  \one_A, \, g\in W^{k,p}(X, \mu)\}.
\end{array}$$
If  $B\subset X$ is not open, its $C_{k,p}$-capacity is the infimum of the $C_{k,p}$-capacities of the open sets that contain $B$. 

Let $k$ be either $1$ or $2$, $p>1$ and let $f\in W^{k,p}(X, \mu)$. Then $f$ is an equivalence class of functions, its elements are called {\em versions} of $f$. There exists a   version $\tilde{f}$  of $f$ that is Borel measurable and $C_{k,p}$-quasicontinuous, namely for each $\eps >0$ there is an open set $A\subset X$ such that $C_{k,p}(A)$  $\leq \eps$ and $\tilde{f}_{|X\setminus A}$ is continuous. Moreover, for every $r>0$, 
\begin{equation}
\label{cap}
C_{k,p}\{ x\in X:\; | \tilde{f}(x)|>r\} \leq \frac{1}{r}\|(I-L_p)^{-1/2} \tilde{f}\|_{L^p(X, \mu)}.
\end{equation}
See e.g. \cite[Thm. 5.9.6]{Boga}. Such 
$\tilde{f}$ is called {\em precise version} of $f$.  Two precise versions of the same $f$ coincide outside a set with null $C_{k,p}$-capacity. 
Moreover if $f\in \cap_{p>1}W^{k,p}(X, \mu)$ there exist Borel versions $\tilde{f}$ of $f$  that are $C_{k,p}$-quasicontinuous for every $p$. In the sequel we shall always consider one of such versions. The results will be independent on the choice of the version. 
 
 If $G:X\mapsto \R$ is any measurable  function, and $\varphi \in  L^1(X, \mu)$ has nonnegative values, the pull-back measure $\varphi \mu\circ G^{-1}$ is defined on the Borel sets $B$ of  $\R$ by
$$(\varphi \mu\circ G^{-1})(B) := \int_{G^{-1}(B)}\varphi(x) \mu(dx), $$
and it is a finite measure. 
If  $\varphi \in L^1(X, \mu)$ attains both positive and negative values,  $\varphi \mu\circ G^{-1}$ defined as above is a signed measure. 

For other aspects of Gaussian capacities and  Sobolev spaces with respect to Gaussian measures we refer to \cite[Ch. 5]{Boga}, \cite{FP1}.

\subsection{Surface measures}
\label{Surface measures}

We recall the definitions of the $1$-codimensional Hausdorff-Gauss measures that will be considered in the sequel. 

If $m\geq 2$, and $F=\R^m$ is equipped with a norm $|\cdot |$,   we define 
$$\theta^F(dx) := \frac{1}{(2\pi)^{m/2}} \exp(-|x|^2/2)H_{m-1}(dx), $$
$H_{m-1}$ being the spherical $m-1$ dimensional Hausdorff measure in $\R^m$, namely 
$$H_{m-1}(A) = \lim_{\delta\to 0} \inf\bigg\{ \sum_{i\in \N}\omega_{m-1}r_i^{m-1}:\;\cup_{i\in \N}B(x_i, r_i)\supset A, \;r_i <\delta \;\forall i \bigg\}$$
where $ \omega_{m-1} = \pi^{(m-1)/2}/\Gamma(1+(m-1)/2)$ is the Lebesgue measure of the unit sphere in $\R^{m-1}$. 
If $X$ is a separable Banach space endowed with a centered nondegenerate Gaussian measure $\mu$, let $H$ be the relevant Cameron--Martin space. For every finite dimensional subspace $F\subset H$ 
we consider  the orthogonal (along $H$) projection on $F$:
$$x\mapsto  \sum_{i=1}^{m}\langle x, f_i\rangle_H f_i ,\quad x\in H, $$
where $\{f_i:\;i=1, \ldots, m\}$ is any orthogonal basis of $F$. Then there
exists a $\mu$-measurable projection $\pi^F$ on $F$, defined in the whole  $X$, that extends it. Its existence is a consequence of   e.g. \cite[Thm. 2.10.11]{Boga}, which states that for every  $i$ there exists a unique  (up to changes on sets with vanishing measure) linear and $\mu$-measurable function $l_i :X\mapsto \R$ that coincides with  $x\mapsto \langle x, f_i\rangle_H$ on $H$. Then we set 
$$\pi^F(x) := \sum_{i=1}^{m}l_i(x)f_i.$$
If $f_i\in Q(X^*)$, $f_i = Q(\hat{f}_i)$ for some $\hat{f}_i\in X^*$, then  $\langle x, f_i\rangle_H = \hat{f}_i(x)$  for every $x\in H$ and the extension is obvious, $l_i(x)  = \hat{f}_i(x)$ for every $x\in X$. In particular, if $X$ is a Hilbert space, 
$l_i(x)  = \langle x, Q^{-1}f_i\rangle_X$. Still
in the case where $X$ is a Hilbert space, it is convenient to choose an orthonormal basis $\{e_k:\;k\in \N\}$ of $X$ made by eigenvectors of $Q$. If $Qe_k = \lambda_k e_k$, the function 
$l_i$ is the $L^2(X, \mu)$ limit of the sequences of cylindrical functions
$$l_i^m(x) :=  \sum_{k=1}^{m}\frac{\langle x, e_k\rangle_X  \langle f_i, e_k\rangle_X}{\lambda_k}, \quad m\in \N, $$
which  is called  $W_{Q^{-1/2}f_i}$ in \cite{DP}. 
If $F$ is spanned by a finite number of elements of the basis $\mathcal V = \{v_k:= \sqrt{\lambda_k}e_k:\;k\in \N\}$ of $H$, say $F= span\,\{ v_1, \ldots, v_m\}$, then 
$$\pi^F(x)  = \sum_{i=1}^{m}\langle x, Q^{-1}v_i\rangle_X v_i = \sum_{i=1}^{m}\langle x, e_i\rangle_Xe_i ,$$
namely $\Pi^F$ coincides with the orthogonal projection in $X$ over the subspace spanned by $e_1, \ldots e_m$. 

Let   $\widetilde{F}$ be  the kernel  of $\pi^F$. We denote by
$\mu^F$  the image measure of $\mu$ on $F$   through $\pi^F$, and by  $\mu_F$ the image measure of $\mu$ on   $\widetilde{F}$  through   $I-\pi^F$. We identify in a standard way $F$ with $\R^m$, namely the element $\sum_{i=1}^{m}x_i f_i\in F$ is identified with the vector $(x_1, \ldots, x_m)\in \R^m$, 
 and we consider the measure $\theta^F$ on $F$. 

We stress that the norm and the associated distance used in the definition of $\theta^F$ are inherited from the $H$-norm on $F$, not from the $X$-norm. For instance, if $X= 
\R^m = F$, then $d H_{m-1} = dS \circ Q^{-1/2}$ where $dS$ is the usual $m-1$ dimensional spherical Hausdorff measure. So, for every Borel set $A$, 
\begin{equation}
\label{thetaF}
\theta^F(A) = \frac{1}{(2\pi)^{m/2}}\int_{Q^{-1/2}(A)} e^{-|y|^2/2}dS. 
\end{equation}
%

In the general case, for any Borel (or, more generally, Suslin) set $A\subset X$ we set 
\begin{equation}
\label{rhoF}
\rho^F(A) := \int_{ \widetilde{F}} \theta^F(A_x)d\mu_F(x), 
\end{equation}
where $A_x  := \{y\in F:\;x+y\in A\}$. By \cite[Prop. 3.2]{Feyel}, the map $F\mapsto \rho^F(A) $ is well defined (namely, the function $x\mapsto \theta^F(A_x)$ is measurable with respect to $\mu_F$) and increasing, i.e. if $F_1\subset F_2$ then $\rho^{F_1}(A) \leq \rho^{F_2}(A)$. This is sketched in \cite{Feyel}, a detailed proof is in \cite[Lemma 3.1]{Diego}.  By the way, this is the reason to choose the spherical Hausdorff measure in $\R^m$: if the spherical Hausdorff measure is replaced by  the usual Hausdorff measure, such a monotonicity condition may fail.

The Hausdorff--Gauss measure of Feyel--de La Pradelle  is defined by 
\begin{equation}
\label{FDP}
\rho(A) := \sup\{ \rho^F(A):\;F\subset H, \;\mbox{\rm finite dimensional subspace} \}
\end{equation}
Similar definitions were considered in \cite{Diego} and \cite{Hino}, respectively,
\begin{equation}
\label{Diego}
\rho_1(A) := \sup\{ \rho^F(A):\;F\subset Q(X^*), \;\mbox{\rm finite dimensional subspace} \},
\end{equation}
\begin{equation}
\label{Hino}
\rho_{{\mathcal V}}(A) := \sup\{ \rho^F(A):\;F\subset H, \;\mbox{\rm spanned by a finite number of elements of}\; {\mathcal V}  \},
\end{equation}
and moreover  in \cite{Hino} it was assumed  ${\mathcal V}\subset Q(X^*)$. 
Of course, $\rho(A) \geq \rho_1(A)$,  $\rho_1(A)\geq \rho_{{\mathcal V}}(A) $ if ${\mathcal V}\subset Q(X^*)$, and $\rho_{{\mathcal V}}(A) $ could depend on the choice of the basis ${\mathcal V}$ of $H$.  In  section \ref{sect:cont} we shall see that if $A$ is contained in a level set of a good function then $\rho(A)= \rho_1(A)= \rho_{{\mathcal V}}(A) $. 

An important property that will be used later is the following (\cite[Thm. 9]{FP}).

\begin{Proposition}
\label{pr:cap}
If $C_{1,p}(A) =0$ for some $p>1$, then $\rho^F(A) =0$ for every $F$, hence $\rho(A)=0$. 
\end{Proposition}

\subsection{Sobolev spaces on sublevel domains}
\label{Sobolev}

In this section $G:X\mapsto \R$ is any  Borel version of an element of $W^{1,q}(X, \mu)$ for some $q>1$, and 
we assume  that  $\O := G^{-1}(-\infty, 0)$ has positive measure. We set as usual $q':=q/(q-1)$.
   
The Sobolev spaces $W^{1,p}(\O, \mu)$ will be defined taking Lipschitz functions as starting points. Let $\varphi\in Lip(\O)$. It is well known that $\varphi$ has a Lipschitz continuous  extension $\widetilde{\varphi}$ to the whole $X$, with the same Lipschitz constant $L$ of $\varphi$. For instance, we can take the McShane extension $\widetilde{\varphi}(x) := \sup\{ \varphi(y) - L\|x-y\|:\;y\in \O\}$. 
Since $Lip(X)\subset W^{1,p}(X, \mu)$ for every $p\geq 1$  (\cite[Ex. 5.4.10]{Boga}), $D_H \widetilde{\varphi}$ is well defined. For any other extension 
$\widetilde{\varphi}_1 \in W^{1,p}(X, \mu)$ for some $p$, we have $D_H \widetilde{\varphi}_{|\O} = D_H \widetilde{\varphi}_{1|\O} $, a.e. in $\O$, by \cite[Lemma 5.7.7]{Boga}. Hence, we can define
$$D_H: \;Lip(\O)\mapsto L^p(\O, \mu; H)$$
as 
$$D_H\varphi := D_H \widetilde{\varphi}_{|\O}, \quad \mbox{\rm for any Lipschitz continuous extension}\;\widetilde{\varphi}\; \mbox{\rm of}\;\varphi.$$

We need the following lemma, about the closability of $D_H$. 
 
\begin{Lemma}
\label{Le:chiusura}
Let $p\geq q'$. Then the operator $D_H: Lip (\O)\mapsto L^p(\O, \mu; H)$ defined above is closable in $L^p(\O, \mu)$. 
\end{Lemma}
\begin{proof}
Let  $f_k \in Lip (\O)$ be such that  $f_k\to 0$ in $L^p(\O, \mu)$ and  $D_Hf_k\to \Phi$ in $L^p(\O, \mu; H)$ as $k\to \infty$. Without loss of generality we may assume that each $f_k$ is defined and Lipschitz continuous in the whole $X$, so that it belongs to $W^{1,q}(X, \mu)$ for every $q>1$. 
We have to  show that  
$\int_{\O} \langle  \Phi, v_i\rangle_H u\,d\mu =0$ for each $i\in \N$ and  $u\in  L^{p'}(\O, \mu)$. Since the restrictions to $\O$ of the Lipschitz continuous functions on $X$ are dense in $ L^{p'}(\O, \mu)$, as a consequence of the density of  $Lip(X)$ in $L^{p'}(X, \mu)$, it is enough to show that 
\begin{equation}
\label{nullo}
\int_{\O} \langle \Phi(x), v_i\rangle_H\, u(x)\,d\mu =0, \quad u\in  Lip(X).
\end{equation}
To this aim we approach every Lipschitz continuous $u$ by functions belonging to $W^{1,p'}(X,\mu)$ that vanish in $\O^c$. 
Fix a smooth $\eta:\R \mapsto \R$ such that  $\eta(r) =0$ for $r\geq -1$, $\eta(r) =1$ for $r\leq -2$, and set   $\eta_n(r):= \eta (nr)$. Then set $u_n(x) := u(x)\eta_n(G (x))$, for each  $n\in \N$ and $x\in X$. By dominated convergence the sequence  $(u_n)$ goes to  $u$ in $L^{p'}(\O, \mu)$ as $n\to \infty$. Therefore 
$$\int_{\O} \langle \Phi(x), v_i\rangle_H  u(x)\,d\mu = \lim_{n\to \infty} \int_{\O}  \langle \Phi(x), v_i\rangle_H u_n(x)\,d\mu .$$
Each  $u_n$ belongs to $W^{1,q}(X, \mu)$, and  $D_iu_n(x) = D_iu(x) \,\eta_n( G(x)) + u(x)\eta_n'(G(x))\,D_iG(x)$ for a.e. $x\in  X$,  for each  $i\in \N$. 

The integration by parts formula  \eqref{partiX}  yields
$$\int_{\O} D_i(f_ku_n) d\mu = \int_{X} D_i(f_ku_n) d\mu =   \int_{X} \hat{v}_i f_ku_n\,d\mu=    \int_{\O}\hat{v}_i f_ku_n\,d\mu$$
so that 
$$\int_{\O} u_n \,D_if_k\,d\mu = -\int_{\O} f_kD_iu_n\,d\mu +   \int_{\O} \hat{v}_i f_ku_n\, d\mu$$
and letting   $k\to \infty$, the left hand side goes to $\int_{\O} u_n \, \langle \Phi, v_i\rangle_H\,  d\mu$ and the right hand side goes to $0$. Therefore $\int_{\O} u_n \, \langle \Phi, v_i\rangle_H\, d\mu =0$ for each  $n$ and  \eqref{nullo} holds. 

The restriction $p\geq q'$ comes from the integral $\int_{\O} f_kD_iu_n\,d\mu$, where $f_k\to 0$ in $L^p(\O, \mu)$ and $D_iu_n\in  L^q(\O, \mu)$. 
 \end{proof}

\begin{Definition}
For $p\geq q'$  the Sobolev space $W^{1,p}(\O, \mu)$ is defined as the domain of the closure of $D_H$ (still denoted by $D_H$) in $L^{p}(\O, \mu; H)$. 
\end{Definition}

$W^{1,p}(\O, \mu)$ is a Banach space with the graph norm 
$$\begin{array}{lll}
\ds \|f\|_{W^{1,p}(\O, \mu)} & := &\ds  \|f\|_{L^p(\O, \mu)} +  \bigg( \int_{\O} |D_Hf(x)|_H^{p}\mu(dx)\bigg)^{1/p}
\\
\\
& =& \ds   \|f\|_{L^p(\O, \mu)} +
\bigg( \int_\O \bigg(\sum_{k=1}^{\infty}   (D_{k}f(x))^2\bigg)^{p/2}\mu(dx)\bigg)^{1/p} , \end{array}$$
where  $D_{k}f(x): = \langle D_Hf(x), v_k\rangle_H$.

Note that the restrictions to $\O$ of the elements of $W^{1,p}(X, \mu)$ belong to $W^{1,p}(\O, \mu)$. Indeed, since $C^1_b(X)$ is dense in $W^{1,p}(X, \mu)$, then each $f\in W^{1,p}(X, \mu)$ may be approached by a sequence of Lipschitz continuous functions, whose restrictions to $\O$ are a Cauchy sequence in 
$W^{1,p}(\O, \mu)$ that converge to $f_{|\O}$ in $L^{p}(\O, \mu)$. Then, $f_{|\O} \in W^{1,p}(\O, \mu)$. 
 
In fact, smaller subspaces consisting of smoother functions are dense in $W^{1,p}(\O, \mu)$, as the next proposition shows.

\begin{Proposition}
\label{approx}
Let $\mathcal D$ be a dense subspace of  $W^{1,p}(X, \mu)$ such that the restrictions to  $\O$ of the elements of  $\mathcal D$ are Lipschitz continuous in $\O$. Then the restrictions to  $\O$ of the elements of  $\mathcal D$ are dense  in $W^{1,p}(\O, \mu)$. 
\end{Proposition}
\begin{proof}
Let  $f\in W^{1,p}(\O, \mu)$ and  $\eps >0$. Let  $\varphi \in Lip (\O)$ be such that  $\|f-\varphi\|_{W^{1,p}(\O, \mu)} \leq \eps$, let $\hat{\varphi}\in Lip (X)$ be  any Lipschitz extension of  $\varphi$ and  let $\psi\in 
 \mathcal D$ be such that 
$\| \varphi -\psi\|_{W^{1,p}(X, \mu)} \leq \eps$. Since   $\psi_{|\O}$ is Lipschitz continuous, then it belongs to $W^{1,p}(\O, \mu)$, moreover  $\| f -\psi_{|\O} \|_{W^{1,p}(\O, \mu)} \leq 2\eps$.
\end{proof}

Note that we can take as $\mathcal D$ the space of the smooth cylindrical functions, as well as the space of the exponential functions  (that is, the span of the functions of the type $x\mapsto e^{i\langle x, h\rangle}$ with $h\in X$) used in \cite{DPZ} when $X$ is a Hilbert space.   
 
As a consequence of Proposition \ref{approx} we get the following lemma, that will be used later. 

\begin{Lemma}
\label{product}
Let $\varphi\in W^{1,p}(\O, \mu)$, $\psi\in W^{1,q}(\O, \mu)$ with $pq/(p+q)>1$ (namely, $p>q'$). Then $\varphi \,\psi\in W^{1,r}(\O, \mu)$ for every $r\in [1, pq/(p+q)]$, and $D_H(\varphi\,\psi) = \psi D_H\varphi + \varphi D_H\psi$. 
\end{Lemma}
\begin{proof} Let $(\varphi_n)$, $(\psi_n)$ be sequences of smooth cylindrical functions whose restrictions to $\O$ converge to $\varphi$, $\psi$, in $ W^{1,p}(\O, \mu)$, $W^{1,q}(\O, \mu)$ respectively. Such sequences exist by Proposition \ref{approx}. As easily seen, $(\varphi_n\, \psi_{n|\O})$ converges to $\varphi \,\psi$ in $L^r(\O, \mu)$, and since $D_H(\varphi_n\psi_n ) = \psi_n D_H\varphi_n+ \varphi_n D_H\psi_n$ for every $n$, then the sequence 
$((D_H(\varphi_n\psi_n ))_{\O} )$ converges to $\psi D_H\varphi + \varphi D_H\psi$ in $L^r(\O, \mu)$. 
\end{proof}


\section{Continuity of densities}
\label{sect:cont}


Our leading assumptions will be the following. 

\begin{Hypothesis}
\label{HypG}
$\hphantom{.}$
\begin{enumerate}

\item $G \in W^{2,q}(X , \mu)$ for each  $q>1$, 

\item $\mu(G^{-1}(-\infty, 0 ) >0$, $G^{-1}(0)\neq \emptyset $,

\item there exists $  \delta >0$ such that $1/|D_HG |_H \in L^{q}(G^{-1}(-\delta, \delta) , \mu)$ for each $q>1$. 
\end{enumerate}
 \end{Hypothesis}

From now on we consider  precise Borel versions of $G$ and $|D_HG|_H$ that we still call $G$ and $|D_HG|_H$. As in  Section \ref{Sobolev}, we consider the set $\O:= G^{-1}(-\infty, 0)$, and for $\delta >0$ we define
$\O_{\delta} := G^{-1}(-\delta, \delta)$.

We use a consequence of the coarea formula \cite[Thm. 5.7]{Feyel}: if $G$ satisfies Hypothesis \ref{HypG}-(1), for each Borel $\psi:X\mapsto [0, +\infty)$ we have
\begin{equation}
\label{coarea}
\int_X \psi |D_HG |_Hd\mu = \int_{\R} \int_{G =\xi} \psi \,d\rho_{{\mathcal V}} \;d\xi .
\end{equation}
(It is not excluded that both members are $+\infty$).

\begin{Lemma}
\label{Cor:coarea}
Let $\varphi :\O_{\delta}\mapsto \R$ be a Borel version of an element of  $L^1(\O_{\delta}, \mu)$, for some $\delta >0$. Then 
the function
\begin{equation}
\label{qvarphi}
q_{\varphi}(\xi):= \int_{G=\xi}\frac{\varphi}{|D_HG|_H}\,d\rho_{{\mathcal V}}  , \quad -\delta <\xi<\delta ,
\end{equation}
belongs to  $L^1(-\delta, \delta)$ and it is a density of the measure $\varphi \mu\circ G^{-1}$ restricted to  $(-\delta, \delta)$.   Moreover, 
\begin{equation}
\label{maggL1} \|q_{\varphi}\|_{L^{1}(-\delta, \delta)} \leq  \|\varphi\|_{ L^1(\O_{\delta}, \mu)}.
\end{equation}
\end{Lemma}

Note that by Proposition \ref{pr:cap} the function $q_{\varphi}$ defined in \eqref{qvarphi} is the same for every precise versions of $G$ and  $|D_HG|_H$. 
\begin{proof} For every Borel set   $B\subset (-\delta, \delta)$ let us consider the function
$$\psi = \frac{ 1}{|D_HG|_H} \varphi  \one_{G^{-1}(B)} .$$
Since  both the positive and the negative parts of  $\psi |D_HG|_H$ are in $L^1(X, \mu)$, we may use  \eqref{coarea}, that yields  $q_{\varphi}\in L^1(B)$ and 
$$\int_{G^{-1}(B)} \varphi \,d\mu = \int_B \int_{G=\xi} \frac{\varphi}{|D_HG|_H}\,d\rho_{{\mathcal V}}\,d\xi.$$

Then, $q_{\varphi}$ is a density of $\varphi \mu\circ G^{-1}$ with respect to the Lebesgue measure  in  $(-\delta, \delta)$. Taking   $B= (-\delta, \delta)$ and applying  \eqref{coarea} to $|\psi|$ we get 
$$ \int_{\R} \bigg| \int_{G=\xi} \psi d\rho_{{\mathcal V}} \bigg| \;d\xi \leq  \int_{\R} \int_{G=\xi} |\psi | d\rho_{{\mathcal V}}  \;d\xi 
= \int_{\O_{\delta}} |\varphi|  d\mu $$
and the estimate follows. 
\end{proof}

For the moment we only know that $q_{\varphi}(\xi)$ is finite for a.e. $\xi \in (-\delta, \delta)$. 
The aim of this section is to prove that  if $\varphi$ is a Borel precise version of an element of $W^{1,p}(X, \mu)$ for some $p>1$, then $ q_{\varphi}(\xi)\in \R$ for every $\xi \in (-\delta, \delta)$,  $ q_{\varphi}$ is continuous in $(-\delta, \delta)$, and moreover $\rho = \rho_{\mathcal V}$ on $G^{-1}(\xi)$, so that 
\begin{equation}
\label{qvarphirho}
 q_{\varphi}(\xi) = \int_{G=\xi}\frac{\varphi}{|D_HG|_H}\,d\rho   , \quad -\delta <\xi<\delta ,
\end{equation}
is independent of the basis $\mathcal V$.  A first step is the Sobolev regularity of $ q_{\varphi}$, which follows from standard arguments, see e.g.  \cite[Ex. 6.9.4]{Boga} or  the appendix of \cite{DPL} in the case that $X$ is a Hilbert space. However, we give the proof for the reader's convenience.

\begin{Proposition}
\label{W11}
Let $p>1$ and let $\varphi$ be a Borel version of an element of $W^{1,p}(X, \mu)$. Then  the signed measure $\varphi \mu\circ G^{-1}$ is absolutely continuous with respect to the Lebesgue measure in the interval $(-\delta, \delta)$,  its  density $q_{\varphi}$ belongs to 
$W^{1,1}(-\delta, \delta)$, and  
\begin{equation}
\label{maggfond} \|q_{\varphi}\|_{W^{1,1}(-\delta, \delta)} \leq C \|\varphi\|_{W^{1,p}(X, \mu)}, 
\end{equation}
with $C$ independent on $\varphi$. 
\end{Proposition}
\begin{proof}
By Lemma \ref{Cor:coarea}, $\varphi \mu\circ G^{-1}$ has density $q_{\varphi}$ with respect to the Lebesgue measure.
We shall show that  $q_{\varphi}$ is weakly differentiable  in  $(-\delta, \delta)$ with $q_{\varphi}' = q_{\varphi_1}$, where
$$\varphi_1 = \di \,\frac{\varphi  D_HG}{|D_HG(x) |^2_H} $$
and $\di$ is the Gaussian divergence. Namely, 
$$\varphi_1 =  \bigg(   \frac{LG }{|D_HG |_H^2}  - 2\,\frac{\langle  D^2_HG \,D_HG, D_HG\rangle_H}{ |D_HG |_H^4}\bigg) \varphi  + \frac{\langle D_HG,  D_H\varphi\rangle_H}{ |D_HG |_H^2} .$$
It will follow that  $q_{\varphi}\in W^{1,1}(\delta, \delta)$ since  $\varphi_1  \in L^1(G^{-1}(-\delta, \delta))$ and by Lemma \ref{Cor:coarea} the density  $q_{\varphi_1}$ of
$\varphi_1\mu \circ G^{-1}$ belongs to  $L^{-1}(-\delta, \delta)$. 

Let  $\eta\in C^{\infty}_{c}(-\delta, \delta)$. Since 
$$ D_H ( \eta\circ G) (x)   = ( \eta' \circ G)(x)D_HG(x) $$
multiplying by  $D_HG(x) $ we get
$$ ( \eta' \circ G)(x) = \frac{  \langle D_H ( \eta\circ G) (x) , D_HG(x)\rangle_H }{|D_HG(x) |^2_H}$$
and replacing
$$\int_{-\delta}^{\delta}\eta'(\xi) q_{\varphi}(\xi)d\xi = \int_{\O_{\delta}} ( \eta' \circ G)(x) 
\varphi(x) \mu(dx) $$
$$= \int_{\O_{\delta}}   \varphi(x) \frac{  \langle D_H ( \eta\circ G) (x) , D_HG(x)\rangle _H}{|D_HG(x) |^2_H}\mu(dx) .$$
The last integral is in fact an integral over $X$, since the support of the integrand is contained in   $\O_{\delta}$. The integrand may be written as  $ \langle D_H ( \eta\circ G), \Psi\rangle_H$ , with  $\Psi = \varphi  D_HG/ |D_HG(x) |^2_H$. By our assumptions,    $\Psi \in L^q(X, \mu; H)$ for every $q>1$, and it belongs to 
$W^{1,q}(X, \mu; H)$ for every  $q<p$, then we may integrate by parts and we get 
$$ \int_{-\delta}^{\delta}\eta'(\xi) q_{\varphi}(\xi)d\xi = - \int_{\O_{\delta}}   ( \eta \circ 
G)(x) \varphi_1(x) \mu(dx) =
  - \int_{-\delta}^{\delta} \eta(\xi) (\varphi_1\,\mu\circ G^{-1})(d\xi).$$
Since  $\varphi_1\in L^1(\O_{\delta})$, by Lemma  \ref{Cor:coarea}  the signed measure $\varphi_1\mu\circ G^{-1}(d\xi)$ has density
$q_{\varphi_1}$, which is the weak derivative of $q_{\varphi}$. Eventually, estimate  \eqref{maggL1} implies \eqref{maggfond}. 
\end{proof}

\begin{Remark}
{\em Note that the assumption $p>1$ is crucial to get  $\varphi_1  \in L^1(\O_{\delta})$ since it is not reasonable to assume that  $ \frac{LG }{|D_HG |_H^2}  - \frac{\langle  D^2_HG \,D_HG, D_HG\rangle_H}{ |D_HG |_H^4}$ and $1/|D_HG |_H$ are bounded in $\O_{\delta}$. Such conditions are  satisfied only in special cases. For instance, if $\O$ is the unit ball in a Hilbert space, $G(x) = \|x\|^2 -1$, $|D_HG(x)|_H = 2\|Q^{1/2}x\|$ so that $1/|D_HG |_H$ is not bounded in any $\O_{\delta}$. This example and other ones will be  treated in 
 Sect. \ref{Examples}. }
\end{Remark}

Since  $q_{\varphi}\in W^{1,1}(-\delta, \delta)$, then there exists a continuous function in $[-\delta, \delta]$ that coincides with  $q_{\varphi}$ almost everywhere. But in the proof of the integration by parts formula \eqref{parti} (Proposition \ref{p5.1}) we need that $q_{\varphi}$ itself is continuous (here we fill a hole in  \cite{BDPT1,BDPT2,DPL}, where this need was neglected).  
 
We shall use the next lemma, whose proof is shrinked  to half a line in \cite{Feyel} and in \cite{FP}. 
In the following we denote by  $D_H^FG$  the orthogonal projection (along $H$) of $D_HG$ on $F$.

\begin{Lemma}
\label{Le:sofferenza}
Let $F$ be a finite dimensional subspace of  $H$, and let $\rho^F$  be defined by \eqref{rhoF}. Then the measures 
$$\frac{d\rho}{|D_H G|_H} ,   \quad \frac{d\rho^F}{|D_H^FG|_H} $$
coincide on $\{ x: G(x) = \xi, \; |D_H^FG|_H\neq 0\}$, for every $\xi \in (-\delta, \delta)$. 
If $F$ is spanned by a finite number of elements of $\mathcal V$, the measures
$$ \frac{d\rho_{\mathcal V}}{|D_H G|_H} , \quad \frac{d\rho^F}{|D_H^FG|_H} $$
coincide on $\{ x: G(x) = \xi, \; |D_H^FG|_H\neq 0\}$, for every $\xi \in (-\delta, \delta)$. 
\end{Lemma}
\begin{proof}
The statement holds if $X$ is finite dimensional, by  \cite[Cor. 6.3]{Feyel}. Consequently, in the infinite dimensional case if $L\supset F$ is a finite dimensional subspace of $X$, for each  $\xi \in (-\delta, \delta)$ the measures
$$\frac{d\rho^L}{|D_H^LG|_H} , \quad \frac{d\rho^F}{|D_H^FG|_H} $$
coincide on  $\{ x: G(x) = \xi, \; |D_H^FG|_H\neq 0\}$. 

Let us prove that $ d\rho/|D_H G|_H= d\rho^F/|D_H^FG|_H  $ on $\{ x: G(x) = \xi, \; |D_H^FG|_H\neq 0\}$. 
Fix $\xi$ such that   $\int_{G=\xi }\frac{d\rho}{|D_H G| }<\infty$. Given any Borel set $A \subset \{x\in X:  \;D_H^FG(x) \neq 0\}$, we have
\begin{equation}
\label{disugL}
\int_{A \cap \{G = \xi\}} \frac{d\rho ^L}{|D_H G|_H}  \leq \int_{A \cap \{G = \xi\}} \frac{d\rho^L}{|D_H^LG|_H} \leq \int_{A \cap \{G  = \xi\}} \frac{d\rho }{|D_H^LG|_H} .\end{equation}
The first equality holds since $|D_H G|_H\geq |D_H^LG|_H$, the second one holds since for each nonnegative Borel function  $\int_A \varphi \,d\rho^L  \leq \int_A \varphi \,d\rho$ by the definition of  $\rho$ as a supremum. 

Now we want to take the sup with respect to $L$.  
Since  $\rho$ is defined as the supremum of  $\rho ^L$, for every nonnegative Borel function $\psi$ we have  $\int_X \psi\,d\rho = \sup_L \int_X \psi\,d\rho^L$. Taking  $\psi = \one_A/|D_H G|_H$  we get
$$ \int_{A \cap \{G = \xi\}} \frac{d\rho  }{|D_H G|_H} =  \sup_{L} \int_{A \cap \{G = \xi\}} \frac{d\rho ^L}{|D_H G|_H}  \leq  \sup_{L} \int_{A \cap \{G = \xi\}} \frac{d\rho ^L}{|D_H^LG|_H}  $$
and recalling that the measures $d\rho ^L/|D_H^LG|_H$ are independent of $L$, 
$$ \int_{A \cap \{G = \xi\}} \frac{d\rho  }{|D_H G|_H} \leq  \int_{A \cap \{G = \xi\}} \frac{d\rho ^L}{|D_H^LG|_H}, \quad \forall L\supset F .$$
In particular, if  $I:= \int_{A \cap \{G = \xi\}} \frac{d\rho  }{|D_H G|} =+\infty$ then  $I_L:=\int_{A \cap \{G = \xi\}} \frac{d\rho ^L}{|D_H^LG|}= +\infty$ for every  $L\supset F$, and in this case the equality $d\rho/|D_H G|_H=d\rho^F/|D_H^FG|_H$ follows. 

If  $I <\infty$ we have to prove also the other inequality.  
Note that  $I <\infty$ does not immediately imply that 
for some   $L\supset F$ we have
$I_L <\infty$. Let us consider the sets $A_n:=\{ x\in A:\; |D_H G(x)|_H/|D_H^LG(x)|_H <n\}$ (recall that $D^F_HG \neq 0$ in $A$, so that $D^L_HG \neq 0$ in $A$). Then  $A_n\subset A_{n+1}$,  and the restriction of the function $1/|D_H^LG(x)|_H$ to  $A_n \cap \{G = \xi\}$ belongs to $L^1(A_n \cap \{G = \xi\},\rho)$, since it is bounded by $n/ |D_H G(x)|_H\one_{G=\xi}$ which belongs to $L^1(A \cap \{G = \xi\}, \rho)$ by assumption. Since  $|D_H^LG(x)|_H$ converges increasingly to  $|D_HG(x)|_H$ as  $L$ increases, by monotone convergence we get
$$\int_{A_n \cap \{G = \xi\}} \frac{d\rho }{|D_H G|_H} = \inf_{L} \int_{A _n\cap \{G = \xi\}} \frac{d\rho }{|D_H^LG|_H}$$
and applying  the second inequality of \eqref{disugL} to $A_n$ we get
$$\inf_L \int_{A _n\cap \{G = \xi\}} \frac{d\rho^L}{|D_H^LG|_H} \leq  \inf_{L} \int_{A _n\cap \{G = \xi\}} \frac{d\rho }{|D_H^LG|} = \int_{A_n \cap \{G = \xi\}} \frac{d\rho }{|D_H G|_H}. $$
Since  $d\rho^L/|D_H^LG|_H$ is constant, 
$$\int_{A _n\cap \{G = \xi\}} \frac{d\rho^L}{|D_H^LG|_H} \leq \int_{A_n \cap \{G = \xi\}} \frac{d\rho }{|D_H G|_H}\quad \forall L\supset F.$$
Letting $n\to \infty$, by monotone convergence in both sides we get 
$$\int_{A  \cap \{G = \xi\}} \frac{d\rho^L}{|D_H^LG|_H} \leq \int_{A  \cap \{G = \xi\}} \frac{d\rho }{|D_H G|_H}\quad \forall L\supset F.$$
Therefore, $d\rho/|D_H G|_H = d\rho^F/|D_H^FG|_H$. The equality $d\rho_{\mathcal V}/|D_H G|_H = d\rho^F/|D_H^FG|_H$ is proved in the same way, just considering only subspaces $L$ spanned by elements of the basis $\mathcal V$.  
\end{proof}

Lemma \ref{Le:sofferenza} has some useful  consequences.

\begin{Corollary}
The measures $\rho$ and $\rho_{\mathcal V}$ coincide on $ G^{-1}(\xi)$, for every $\xi\in (-\delta, \delta)$. 
\end{Corollary}
\begin{proof}
Lemma \ref{Le:sofferenza} implies that for every  $\xi \in (-\delta, \delta)$ the measures 
$\rho/|D_HG|_H$,  $\rho_{\mathcal V}/|D_HG|_H$ coincide with  $\rho^L/|D_H^LG|_H$ on $\{  x\in X:\  \; D_H^L G(x)\neq 0, \; G(x)=\xi \}$, where $L$ is any finite dimensional subspace  spanned by elements of the basis $\mathcal V$. 
Then, $\rho/|D_HG|_H$,  $\rho_{\mathcal V}/|D_HG|_H$ coincide on the union of such sets, which is just $\{ x\in X:\  \; D_H G(x)\neq 0, \;G(x) = \xi \}$. 

We remark that  Hypothesis \ref{HypG} implies that the set $\{ x\in X:\  \; D_H G(x) = 0\}$ has null $C_{1,p}$-capacity, for every $p$. Indeed, it is sufficient to apply estimate \eqref{cap} to the function  $f= 1/|D_HG|_H$ (that belongs to $W^{1,p}(X, \mu)$ for every $p$) and to observe that $\{ x\in X:\  \; D_H G(x) =  0\}$ $\subset$ $\{ x\in X:\  \; 1/|D_H G(x)|>r\}$ for every $r>0$. 
By Proposition \ref{pr:cap},  $\rho(A) =  \rho_{\mathcal V}(A) =0$ for every set $A$ with null $C_{1,p}$-capacity. Then $\rho/|D_HG|_H$,  $\rho_{\mathcal V}/|D_HG|_H$ coincide on $G^{-1}(\xi)$ for every $\xi\in (-\delta, \delta)$, and the  conclusion follows. 
\end{proof}

\begin{Corollary}
\label{Cor:sofferenza}
For each Borel precise  $\varphi\in  W^{1,p}(X, \mu)$ such that the support of $\varphi_{| \O_{\delta}}$ is contained  in $\{x\in \O_{\delta}:\;D_H^FG(x) \neq 0\}$ for some $F$ we have
$$\int_{G=\xi} \varphi  \frac{d\rho  }{|D_H G|_H} = \int_{G=\xi} \varphi  \frac{d\rho _{\mathcal V} }{|D_H G|_H} = \int_{G=\xi} \varphi  \frac{d\rho^F}{|D_H^FG|_H}, \quad -\delta < \xi < \delta $$
and consequently 
\begin{equation}
\label{sofferenza}
q_{\varphi}(\xi) =   \int_{\widetilde{F}} q_{\varphi_x}(\xi)  \mu_F(dx), \quad -\delta <\xi <\delta . 
\end{equation}
\end{Corollary}
\begin{proof} 
The  statement is obtained just integrating with respect to the measures  $d\rho_{\mathcal V}/|D_H G|_H $ $=$ $ d\rho^F/|D_H^FG|_H$. \end{proof}
 
With the aid of Corollary \ref{Cor:sofferenza} we may eventually prove that $q_{\varphi}$ is continuous.

\begin{Theorem}
\label{cont}
For every Borel precise  $\varphi \in W^{1,p}(X, \mu)$ with $p>1$, we have  
$$q_{\varphi}(\xi)= \int_{\{G=\xi\}} \frac{\varphi}{|D_HG|_H}\,d\rho \in \R, \quad \xi\in (-\delta, \delta),$$
and $q_{\varphi}$ is continuous  in $(-\delta, \delta)$. 
\end{Theorem}
\begin{proof}
We follow (and expand) the arguments of \cite{Feyel}. 

\vspace{3mm}
\noindent {\em Step 1.} As a  first  step we consider the case where $X$ is finite dimensional. Let  $\varphi$ be a $C^1$ function with compact support in $\O_{\delta}$. In this case $G$ is $C^1$, the level surfaces  $\{ G=\xi \}$ are $C^1$ for every  $\xi$ in $(-\delta, \delta)$, and recalling \eqref{thetaF} at each level surface we have
$$d\rho = \frac{e^{-\langle Q^{-1}x, x\rangle/2}}{ (\Det Q)^{1/2}(2\pi)^{m/2}} \frac{\|Q^{1/2}DG(x)\|}{\|DG(x)\|}\,H_{m-1}(dx)$$
if $m$ is the dimension of $X$ (here we have considered the usual scalar product and norm). Since the level surfaces have $C^1$ parametrizations and  the boundary integrals are surface integrals with  weight, then $q_{\varphi}$ depends continuously on  $\xi$. 

If  $\varphi \in W^{1,p}(X, \mu)$ is precise and has compact support,  
it is approached by a sequence of smooth $\varphi_n$ with compact support (the usual sequence of convolutions with standard mollifiers does the job). The restrictions of $\varphi_n $ to the surface $\{ G=\xi \}$ converge in $L^p(\{G=\xi\}, \rho)$ to the trace of $\varphi  $ at $\{ G=\xi \}$. Note that on every compact set the Gaussian $L^p$ and Sobolev spaces are equivalent to $L^p$ and Sobolev spaces with respect to the Lebesgue measure, and the trace of $\varphi$ at $\{G=\xi\}$ is well defined (e.g., \cite[Sect. 4.3]{EG}). Moreover by estimate  \eqref{maggfond} the sequence $q_{\varphi_n}$ is a Cauchy sequence in $L^{\infty}$, so that it converges in the sup norm (since each $q_{\varphi_n}$ is continuous). Then the  pointwise limit  of $q_{\varphi_n}$ is in fact a uniform limit, so that it is continuous in  $(-\delta, \delta)$. 

To identify such pointwise limit with $q_{\varphi}$ we remark that the trace at $\{G=\xi\}$ of $\varphi $ 
 coincides $\rho$-a.e. with the restriction to $\{G=\xi\}$ of any precise version of $\varphi$. This is because for $H_{n-1}$-almost every $x\in G^{-1}(\xi)$ both of them are equal to 
$$\lim_{r\to 0} \frac{1}{|B(x, r)| } \int_{B(x,r)}\varphi(y)dy,  $$
where $|\cdot |$ denotes the Lebesgue measure. The above formula may be easily deduced from e.g. \cite[Sect. 5.3]{EG}. 
 
Therefore, $\lim_{n\to\infty}q_{\varphi_n}(\xi) = q_{\varphi}(\xi)$, for each $\xi\in (-\delta, \delta)$.

If  $\varphi \in W^{1,p}(X, \mu)$ is precise, nonnegative  and has not compact support, for every $\varepsilon \in (0, \delta)$
it may be approached in $W^{1,p}(\O_{\delta -\epsilon})$ by a sequence of   functions $\varphi_n$ with compact support in $\O_{\delta}$,   that converge to $\varphi $ increasingly.    
Then, $\lim_{n\to\infty}q_{\varphi_n}(\xi) = q_{\varphi}(\xi)$ for every $\xi\in (-\delta + \epsilon, \delta - \epsilon)$ by monotone convergence. As before, by estimate  \eqref{maggfond} the sequence $q_{\varphi_n}$ is a Cauchy sequence in $L^{\infty}$, and  it converges in the sup norm, so that  the  pointwise limit $q_{\varphi}$ of $q_{\varphi_n}$ is  a uniform limit and  it is continuous in $(-\delta + \epsilon, \delta - \epsilon)$. Since $ \epsilon $ is arbitrary,  $q_{\varphi}$ is continuous in $(-\delta, \delta)$. 

If $\varphi $ attains both positive and negative values, we write it as the difference between its positive and negative parts $\varphi^+$ and $\varphi^-$, then the equality $q_{\varphi} = q_{\varphi^+} - q_{\varphi^-}$ yields that $q_{\varphi} $ is continuous.

Note that without the assumption  $D_HG\neq 0$, that in finite dimensions is equivalent to $1/|D_HG |_H \in L^{q}_{loc}(\O_{\delta} , \mu)$, such continuity properties still hold for functions $\varphi$ that vanish  in $\{ x\in \O_{\delta}:\; |D_HG|\leq \varepsilon\}$ for some $ \varepsilon >0$. 
 
 \vspace{3mm}
 
{\em Step 2.} Let $X$ be infinite dimensional. Consider any finite dimensional subspace $F\subset H$, the orthogonal (along $H$) projection on $F$, and its extension $\Pi^F$ to $X$ mentioned in \S \ref{Surface measures}. We recall that $\widetilde{F}$ is  the kernel  of $\pi^F$ and $\mu^F$, $\mu_F$ are  the image measures of $\mu$ on $F$, $\widetilde{F}$   through $\pi_F$,    $I-\pi^F$ respectively. Fix any Borel precise $\varphi \in W^{1,p}(X, \mu)$ that vanishes at 
$\{ x\in F:\; |D_H^F G|_H \leq \varepsilon \}$ for some $ \varepsilon>0$. 

For every $x\in \widetilde{F}$, we consider the subset  $\O_{\delta}^x$ of $F$ defined by 
$$\O_{\delta}^x:= \{y\in F:\; G(x+y)\in (-\delta, \delta)\}$$ 
(it may be empty for some $x$) and the section  $\varphi_x$ defined in $\O_{\delta}^x$ by $\varphi_x(y) := \varphi(x+y)$. 

By \cite[Thm. 4.5]{Feyel}, for $\mu_F$-almost all $x\in \widetilde{F}$ the section $G_x$ is a precise element of $\cap_{q>1}W^{1,q}(F, \mu^F)$, hence it is continuous and $\O_{\delta}^x$ is open. By the same theorem, $\varphi_x$ is a precise element of $W^{1,p}(F, \mu^F)$. Moreover, for every $x\in \widetilde{F}$, $\varphi_x$ vanishes in a neighborhood of the zeroes of $D_H^FG_x$, namely in the set $\{y\in F:\; |D_H^FG(x+y)|_H \leq \eps\}$. 
By Step 1,  its density
$q_{\varphi_x}$ (with $G$ replaced by $y\mapsto G_x(y) = G(x+y)$) is continuous. Moreover, by Corollary \ref{Cor:sofferenza},  
\begin{equation}
\label{sezione}
q_{\varphi}(\xi) = \int_{X} q_{\varphi_x}(\xi) (I-\pi^F)(d\mu) = \int_{\widetilde{F} } q_{\varphi_x}(\xi) (I-\pi^F)(d\mu), 
\end{equation}
for every $\xi\in (-\delta, \delta)$. Then the statement follows easily:   the function $q_{\varphi_x}$ is continuous in $(-\delta, \delta)$, so we may let $\xi\to \xi_0 \in (-\delta, \delta)$ and use the dominated convergence theorem, since for $\mu_F$-almost each $x\in \widetilde{F}$ and for each  $\xi \in (-\delta, \delta)$ we have
$$q_{\varphi_x}(\xi) \leq \| q_{\varphi_x}\|_{\infty} \leq C(\delta) \|q_{\varphi_x}\|_{W^{1,1}(-\delta, \delta)} $$
$$\leq C(\delta) C \|\varphi_x\|_{W^{1,p}(\O_{\delta}, \mu^F)} \leq   C(\delta) C \|\varphi \|_{W^{1,p}(\O_{\delta}, \mu)} .$$
where $C$ is the constant in formula formula \eqref{maggfond}. 

Now we consider a Borel nonnegative precise $\varphi \in W^{1,p}(X, \mu)$ with any support. Fix any ordering of the basis $\mathcal V$ and 
denote by $F_n$ the subspace generated by the first $n$ elements of  $\mathcal V$. 
There exists a sequence of functions $\varphi_n\in W^{1,p}(X, \mu)$ that converges increasingly to $\varphi $ in $ W^{1,p}(\O_{\delta}, \mu)$, such that each $\varphi_n$ is Borel, precise, and vanishes in $\{x\in \O_{\delta}:\; |D_H^{F_n}G|_H\leq 1/n\}$ (Lemma \ref{Le:approssimazione}). By the first part of the proof, the corresponding densities $q_{\varphi_n}$ are continuous, and by Corollary \ref{Cor:sofferenza} we have
$$q_{\varphi_n}(\xi) =  \int_{\{G=\xi\}} \frac{\varphi_n}{|D_HG|_H}d\rho _{\mathcal V} =  \int_{\{G=\xi\}} \frac{\varphi_n}{|D_HG|_H}d\rho .$$
By monotone convergence, for each $\xi$ we have $\lim_{n\to \infty}q_{\varphi_n}(\xi) =  q_{\varphi}(\xi) =  \int_{\{G=\xi\}} \varphi /|D_HG|_H\,d\rho $. Moreover applying estimate \eqref{maggfond} to $\varphi_n-\varphi_m$ yields that the sequence $q_{ \varphi_n}$ converges in $L^{\infty}$ and hence uniformly, since all of them are continuous functions. Therefore, the pointwise limit $q_{\varphi}$ is in fact a uniform limit, hence it is continuous. 
 
If $\varphi $ takes both positive and negative values, the statement follows by splitting it as $\varphi^+ - \varphi^-$. 
\end{proof}

\begin{Lemma}
\label{Le:approssimazione}
Let $\mathcal V = \{ v_k:\;k\in \N\}$, and set $  F_n =$  span $\{v_1, \ldots, v_n\}$. 
For each  $\varphi \in W^{1,p}(X, \mu)$ there exists a sequence of functions $\varphi_n\in W^{1,p}(X, \mu)$ whose restrictions to $\O_{\delta}$ converge to $\varphi $ in $ W^{1,s}(\O_{\delta}, \mu)$ for every $s<p$, 
and such that each $\varphi_n$ vanishes in $\{x\in \O_{\delta}:\; |D_H^{F_n}G|_H\leq 1/n\}$.  If $\varphi $ is Borel and precise, the functions $\varphi_n$ are Borel and precise too. 
\end{Lemma}
\begin{proof}
Let  $\theta: \R\mapsto \R$ be a smooth function such that $\theta(\xi) = 0$ for $0\leq \xi\leq 1$, $\theta(\xi) = \xi-1$ for $1\leq \xi\leq 2$, $\theta(\xi) = 1$ for $\xi \geq 2$, and set 
$$\varphi_n(x) = \varphi(x) \theta (n|D_H^{F_n}G(x)|_H), \quad x\in X. $$
Then $\varphi_n$ vanishes if $|D_H^{F_n}G(x)|_H\leq 1/n$,   
$\varphi_n\to \varphi $ in  $L^{p}(\O_{\delta},\mu)$ by dominated convergence, and moreover
$$D_H\varphi_n = D_H\varphi(x) \theta (n|D_Hf(x)|_H) + n \varphi(x)\theta'(n|D_H^{F_n}G(x)|_H) D_H(|D_H^{F_n}G(x)|_H).$$
The first term goes to   $D_H\varphi_n$ in  $L^{p}(\O_{\delta},\mu)$, still by dominated convergence. We have to show that the second term vanishes as $n\to \infty$. 

 We have
$$|D_H^{F_n}G(x)|_H = \bigg(\sum_{i=1}^{n} D_iG(x)^2\bigg)^{1/2}$$
and for each  $k\in \N$
$$|D_k(|D_H^{F_n}G(x)|_H)| = \frac{|\sum_{i=1}^{n} D_iGD_{ik}G(x)|}{|D_H^{F_n}G(x)|_H} \leq 
( \sum_{i=1}^{n} (D_{ik}G(x))^2)^{1/2}$$
so that 
$$|D_H(|D_H^{F_n}G(x)|_H)|_H \leq ( \sum_{i,k=1}^{\infty} (D_{ik}G(x))^2)^{1/2} = |D^2_HG(x)|_{{\mathcal H}_2}. $$
Setting $A_n = \{x\in \O_{\delta}:\, |D_H^{F_n}G(x)| \geq 1/n\}$, we consider the integral
$$\int_{A_n}| \varphi(x) D_H(|D_H^{F_n}G(x)|_H)|_H^s\mu(dx) $$
with  $s\leq p$ (note that on the complement of $A_n$ we have $\theta'(n|D_H^{F_n}G(x)|_H)=0$). Using the H\"older inequality we get
$$\int_{A_n}| \varphi(x) D_H(|D_H^{F_n}G(x)|_H)|_H^s\mu(dx) $$
$$\leq \|\varphi\|_{L^p(\O_\delta)} \bigg(\int_{A_n}| D_H(|D_H^{F_n}G(x)|_H)|_{H}^{p/(p-s)}\mu(dx) \bigg)^{(p-s)/p}$$
$$ \leq  \|\varphi\|_{L^p(\O_\delta)} \bigg( \int_{A_n}  |D^2_HG(x)|_{{\mathcal H}_2}^{p/(p-s)}\mu(dx) \bigg)^{(p-s)/p} .
$$
Using the assumptions that $|D^2_HG(x)|_{{\mathcal H}_2}^{p/(p-s)}$ and $ |D_H G(x)|_H^q$ are  $L^2$ functions, 
 for $q\geq 1$ we get
$$\int_{A_n}  |D^2_HG(x)|_{{\mathcal H}_2}^{p/(p-s)}\mu(dx)  = \int_{A_n}  |D^2_HG(x)|_{{\mathcal H}_2}^{p/(p-s)}\frac{|D_H^{F_n}G(x)|_H^q}{|D_H^{F_n}G(x)|_H^q}\mu(dx) $$
$$\leq \frac{1}{n^q} 
\int_{A_n}  |D^2_HG(x)|_{{\mathcal H}_2}^{p/(p-s)} |D_H G(x)|_H^q \mu(dx)
\leq \frac{C_q}{n^q}$$
so that, taking  $q >ps/(p-s) $ we find
$$n \bigg(\int_{\O_{\delta}} |\varphi(x)\theta'(n|D_H^{F_n}G(x)|_H) |D_H(|D_H^{F_n}G(x)|_H)|_{H}^s\mu(dx) \bigg)^{1/s}$$
$$\leq 
n \bigg(\int_{A_n} | \varphi(x)| |D_H(|D_H^{F_n}G(x)|_H)|_{H}^s\mu(dx) \bigg)^{1/s}\leq 
\frac{C_q^{(p-s)/ps}}{n^{q(p-s)/ps-1}}(\|\varphi\|_{L^p(\O_{\delta})})^{1/s}$$
which vanishes as $n\to \infty$. 
\end{proof}

\section{Traces on level surfaces}
\label{sect:traces}

Throughout the section we assume that Hypothesis \ref{HypG} holds. 
Let us state the integration by parts formula and estimates that are the starting point for our study of traces.

\begin{Proposition}
\label{p5.1}
Let $p>1$. Then for every Borel precise  $\varphi\in W^{1,p}(X, \mu)$ and for each  $k\in \N$, \eqref{parti} holds. 
Moreover, 
\begin{equation}
\label{partitraccia}
  \int_{\{G=0\}} | \varphi|^q |D_HG|_H  \,d\rho  =  q   \int_{\O}  | \varphi|^{q-2}\varphi \langle D_H\varphi, D_HG\rangle_H   \,d\mu +  \int_{\O} LG \,| \varphi|^q \,d\mu , 
  \end{equation}
and 
\begin{equation}
\label{partitraccia2}
 \int_{\{G=0\}} | \varphi|^q   \,d\rho = \ds  q   \int_{\O}  | \varphi|^{q-2}\varphi \frac{\langle D_H\varphi, D_HG\rangle_H }{|D_HG|_H }  \,d\mu + \int_{\O} \di \bigg(\frac{D_HG}{|D_HG|_H}\bigg)  \,| \varphi|^q \,d\mu 
\end{equation}
for every $q\in [1,p)$. 
\end{Proposition}
\begin{proof} For $\eps >0$ we define a function $\theta_{\eps}$ by 
$$\theta_{\eps}(\xi) :=\left\{ \begin{array}{ll} 
1, & \xi\leq  -\eps, 
\\
-\frac{1}{ \eps} \xi, &  -\eps <\xi < 0 , 
\\
0, & \xi \geq 0 .
\end{array}\right.  $$
and we consider the function
$$ x\mapsto  \varphi(x)  \theta_{\eps}(G(x)), $$
which belongs to $W^{1,q}(X, \mu)$ for each $q <p$, and its derivative along  $v_k$ is  $  \theta_{\eps}'(G(x))D_k G(x)\varphi(x) $ $+$ $  \theta_{\eps}(G(x))D_k \varphi(x)$. Applying the integration by parts formula \eqref{partiX} we get 
\begin{equation}
\label{prima}
\int_X (D_k\varphi)  (\theta_{\eps}\circ G)  \,d\mu   -\frac{1}{ \eps} \int_{ -\eps<G <0} \varphi D_k G \,d\mu =  
\int_X \hat{v}_k \varphi (  \theta_{\eps}\circ G)\,d\mu, \quad k\in \N.
\end{equation}
Let us prove that   \eqref{parti} holds. As $\eps\to 0$, $\theta_{\eps}\circ G$ converges pointwise to  $ \one_{\O}$. Since  $ \theta_{\eps}\circ G \leq 1$, by dominated convergence we get
$$\exists   \lim_{\eps \to 0} \frac{1}{  \eps} \int_{ -\eps<G <0} \varphi D_k G \,d\mu  = \int_{\O} D_k\varphi \,d\mu -   \int_{\O} \hat{v}_k\varphi \,d\mu .$$
Let us identify this limit as a surface integral. Using the notation of section \ref{sect:cont},   we have
 $\int_{ -\eps<G <0}  \varphi D_k G \,d\mu 
=  \int_{-\eps}^0 q_{ \varphi D_k G}(\xi)d\xi  $. 
Since   $\varphi D_kG  $ belongs to   $W^{1,q}(X, \mu)$ for every $q<p$ and it is Borel measurable and precise, by Theorem \ref{cont} the function $q_{ \varphi D_k G}$ is continuous at $0$. Then, 
$$\lim_{\eps \to 0}  \frac{1}{ \eps} \int_{ -\eps<G <0}    \varphi D_k G \,d\mu  =   q_{ \varphi D_k G}(0)
=
\int_{\{G=0\} } \frac{D_kG}{ |D_HG|_H}  \varphi   \,d\rho $$
and \eqref{parti} follows. 

Now let us prove that   \eqref{partitraccia} holds. For every  $k$, the function  $( \theta_{\eps}\circ G)| \varphi |^qD_kG$ belongs to  $W^{1,q} (X, \mu)$ for every $q<p$. We may replace  $\varphi$ by $  |\varphi|^qD_kG$ in \eqref{prima}, and sum over $k$, since the series  $ \sum_{k=1}^{n}  D_{kk}G(x) -  \hat{v}_kD_kG(x) $ converges to  $LG$ in  $L^r(X, \mu)$ for each $r$ (e.g., \cite[Prop. 5.8.8]{Boga}). We obtain
$$\int_X q|\varphi|^{q-2}\varphi  \langle D_H\varphi, D_HG\rangle_H  ( \theta_{\eps}\circ G) \,d\mu + \int_X  LG \,|\varphi|^q  ( \theta_{\eps}\circ G) \,d\mu $$
$$= \frac{1}{ \eps}  \int_{-\eps<G <0}  |\varphi|^q |D_HG|_H^2  \,d\mu .$$
Proceeding as in the proof of  \eqref{parti}, as $\eps\to 0$ by dominated convergence we get
$$\lim_{\eps \to 0}  \int_X q|\varphi|^{q-2}\varphi  \langle D_H\varphi, D_HG\rangle_H  ( \theta_{\eps}\circ G )\,d\mu  = 
  \int_{\O} q|\varphi|^{q-2}\varphi \langle D_H\varphi, D_HG\rangle_H  \,d\mu, $$
$$\lim_{\eps \to 0}  \int_X LG \,|\varphi|^q (\theta_{\eps}\circ G ) \,d\mu =    \int_{\O} LG \,|\varphi|^q    \,d\mu .$$
Then, there exists the limit
$$\lim_{\eps \to 0}  \frac{1}{ \eps} \int_{ -\eps<G<0} | \varphi|^q |D_HG|_H^2  \,d\mu =  \int_{\O} q|\varphi|^{q-2}\varphi  \langle D_H\varphi, D_HG\rangle_H  \,d\mu +  \int_{\O} LG \,|\varphi|^q   \,d\mu $$
that we identify as before with  a surface integral. Indeed,  since   $\psi:=| \varphi|^q  |D_HG|_H^2 \in W^{1,q} (X, \mu)$ for each  $q<p$ and it is Borel and precise, by Theorem \ref{cont} the density $q_{\psi}$ of $\psi \mu\circ G^{-1}$ with respect to the Lebesgue measure is continuous at $0$, and we get
$$\lim_{\eps \to 0}  \frac{1}{ \eps} \int_{ -\eps<G<0} | \varphi|^q |D_HG|_H^2  \,d\mu = q_{\psi}(0) = \int_{\{G=0\}} | \varphi|^q |D_HG|_H \,d\rho.$$
To prove \eqref{partitraccia2} we follow the same procedure, replacing   $\varphi $ in \eqref{prima} by $  |\varphi|^qD_kG/|D_HG|_H$,  and summing over  $k$. Then, we show that there exists the limit 
$\lim_{\eps \to 0}  \frac{1}{ \eps} \int_{ -\eps<G<0} | \varphi|^q |D_HG|_H   \,d\mu $ and we identify it with the surface integral $\int_{\{G=0\}} | \varphi|^q  \,d\rho$. We obtain
$$
\begin{array}{lll}
\ds  \int_{\{G=0\}} | \varphi|^q   \,d\rho  &= & \ds  q  \int_{\O}  | \varphi|^{q-2}\varphi \frac{\langle D_H\varphi, D_HG\rangle_H }{|D_HG|_H }  \,d\mu +  \int_{\O} \frac{LG}{|D_HG|_H } \,| \varphi|^q \,d\mu 
\\
\\
& -  &\ds  \int_{\O} \frac{\langle D^2_HG\, D_HG, D_HG\rangle_H}{|D_HG|_H^3 } \,| \varphi|^q \,d\mu 
\end{array}
$$
which coincides with \eqref{partitraccia2}, since 
$$\di \bigg( \frac{D_HG}{|D_HG|_H} \bigg) = \frac{LG}{|D_HG|_H} - \frac{\langle D^2_HG\, D_HG, D_HG\rangle_H}{|D_HG|_H^3}.$$
\end{proof}

\begin{Corollary}
\label{MaggTraccia}
For each  $p>1$ and  $\varphi \in W^{1,p}(\O, \mu)$ there exists $\psi \in \cap_{q<p}L^q(\{G=0\},  \rho)$ with the following property: if  $(\varphi_n )\subset   Lip (X)$ are such that $(\varphi_{n|\O})$ converge to $\varphi$ in   $W^{1,p}(\O, \mu)$, the sequence $(\varphi_{n | \{G=0\}} )$ converges to  $\psi$  in $L^q(\{G=0\}, \rho )$, for every $q<p$.
In addition, if either
 \begin{equation}
 \label{strana}
\mu-\mbox{\rm ess sup}_{x\in  \O} \;\;\di \bigg( \frac{D_HG}{|D_HG|_H} \bigg)<+\infty
  \end{equation}
or
\begin{equation}
 \label{strana1}
\left\{ \begin{array}{l}
\mu-\mbox{\rm ess sup}_{x\in  \O} \; |D_HG|_H <+\infty, \;\mu-\mbox{\rm ess sup}_{x\in \O}\;
 LG   <+\infty, 
 \\
 \\
\rho-\mbox{\rm ess inf}_{x\in  G^{-1}(0)} |D_HG(x)|_H  >0, 
 \end{array}\right. 
 \end{equation}
then $\varphi_{n|\{G=0\}}$ converges in $L^p(\{G=0\}, \rho )$.
\end{Corollary}
\begin{proof}  Let us use estimate \eqref{partitraccia2} for the functions $\varphi_n-\varphi_m$.
For every $q\geq 1$ we get 
$$\begin{array}{l}
\ds  \int_{\{G=0\}} | \varphi_n-\varphi_m|^q   \,d\rho  = \ds  q   \int_{\O}  | \varphi_n-\varphi_m|^{q-2}\varphi \frac{\langle D_H(\varphi_n-\varphi_m),  D_HG\rangle_H }{|D_HG|_H }  \,d\mu 
\\
\\ + \ds \int_{\O}\;\di \bigg( \frac{D_HG}{|D_HG|_H} \bigg)\,|\varphi_n-\varphi_m|^q \,d\mu 
\end{array}
 $$
and since  $\d \,D_HG/|D_HG|_H$ belongs to  $L^{r}(\O, \mu)$ for every $r$,  if $q<p$ the H\"older inequality yields that the sequence  $(\varphi_{n|\{G=0\}} )$ is a  Cauchy sequence in $L^q(\{G=0\},  \rho)$, so that it converges to a function $\psi \in L^q(\{G=0\},  \rho)$. 
 
Still by estimate \eqref{partitraccia2}, the limit $\psi$ is the same for all sequences $(\varphi_n)\in  Lip (\O)$ that converge to $\varphi$ in   $W^{1,p}(\O, \mu)$, and it is independent of $q$.

\vspace{2mm}

If \eqref{strana} holds, the above procedure works for $q=p$ too,   without need of the H\"older inequality. If  \eqref{strana1} holds we proceed in the same way, using \eqref{partitraccia} with $q=p$, instead of \eqref{partitraccia2}. 
\end{proof}

Proposition \ref{p5.1} and its corollary allow to define the traces at $G=0$ of the elements of $W^{1,p}(\O, \mu)$. 

\begin{Definition} For each  $\varphi \in W^{1,p}(\O, \mu)$, $p>1$,  we define the trace $\Tr \varphi$ of $\varphi$ at  $\{G=0\}$ as the function $\psi$ given by Corollary
\ref{MaggTraccia}. 
\end{Definition}

By Corollary \ref{MaggTraccia}, the trace operator is bounded from 
$W^{1,p}(\O, \mu)$ to  $L^{q}(\{G=0\}, \rho)$ for each  $q \in [1, p)$. If in addition  \eqref{strana} or  \eqref{strana1} hold, it is bounded from 
$W^{1,p}(\O, \mu)$ to  $L^{p }(\{G=0\}, \rho)$. 

Moreover we may extend formulae  \eqref{parti}, \eqref{partitraccia}  and  \eqref{partitraccia2} 
to all elements of  $W^{1,p}(\O, \mu)$.

\begin{Corollary}
\label{Cor:partiW1p}
For every  $\varphi \in W^{1,p}(\O, \mu)$, $p>1$, we have
\begin{equation}
\label{partiW1p}
\int_{\O} D_k\varphi \,d\mu =   \int_{\O} \hat{v}_k\varphi \,d\mu + \int_{G^{-1}(0)} \frac{D_k G}{|D_HG|_H }\Tr \varphi \,d\rho , \quad k\in \N .
\end{equation}
Moreover, formulae  \eqref{partitraccia} and \eqref{partitraccia2} hold for every $q\in [1,p)$, with   $\Tr \varphi$ replacing $\varphi$ in the  surface integrals.
\end{Corollary}
\begin{proof} It is sufficient to use  \eqref{parti} (respectively, \eqref{partitraccia}, \eqref{partitraccia2}) for any sequence of Lipschitz continuous functions that converge to $\varphi \in W^{1,p}(\O, \mu)$, and take  the limit. \end{proof}

\begin{Remark}
Taking into account formulae  \eqref{partitraccia} and \eqref{partitraccia2} (that are equalities, not estimates), we see that the assumption $G\in W^{2,p}(X, \mu)$ for every $p$ is not very restrictive, since the right hand sides contain second order derivatives of $G$. 
\end{Remark}

Two natural questions arise. The first one is whether the trace operator is bounded from $W^{1,p}(\O, \mu)$ to  $L^p(\{G=0\}, \rho)$ under the only hypothesis \ref{HypG} or under weaker assumptions  than \eqref{strana} or \eqref{strana1}, the second one is whether the traces enjoy  some further regularity properties, as in the finite dimensional case. The problem of the characterization of the range of the trace operator seems to be out of hope for the moment. However, in a very special case (\S 5.1) this characterization is available. 

To get a positive answer to the first question, 
assumptions \eqref{strana} and \eqref{strana1} may be a little weakened. 
 
\begin{Lemma}
Assume that 
\begin{equation}
 \label{strana1bis}
\left\{ \begin{array}{l}
\mu-\mbox{\rm ess sup}_{x\in  G^{-1}(-\delta, 0)} \; |D_HG|_H <+\infty, \;\mu-\mbox{\rm ess sup}_{x\in  G^{-1}(-\delta, 0)}\;
 LG   <+\infty, 
 \\
 \\
 \rho -\mbox{\rm ess  inf}_{x\in  G^{-1}( 0)} \; |D_HG|_H >0
 \end{array}\right. 
 \end{equation}
or that
 \begin{equation}
 \label{stranabis}
\mu-\mbox{\rm ess sup}_{x\in G^{-1}(-\delta, 0)} \;|D_HG|_H <+\infty, \;
\mu-\mbox{\rm ess sup}_{x\in G^{-1}(-\delta, 0)} \;\;\di \bigg( \frac{D_HG}{|D_HG|_H} \bigg)<+\infty
  \end{equation}
for some $\delta >0$. Then the trace operator is bounded from $W^{1,p}(\O, \mu)$ to $L^p(G^{-1}(0), \rho)$, for every $p>1$. 
\end{Lemma}
\begin{proof} It is sufficient to show that there exists $C>0$ such that 
\begin{equation}
\label{maggbuona}
\int_{\{G=0\}} | \varphi|^p   \,d\rho \leq C \|\varphi\|_{W^{1,p}(\O, \mu)}^p  ,
\end{equation}
for every Lipschitz continuous $\varphi$. 

Let $\theta\in C^{\infty}(\R)$ be such that $\theta(\xi) = 1$ for $|\xi |\leq \delta/2$, $\theta(\xi) =0$ for $|\xi |\geq \delta$. 
For any Lipschitz continuous $\varphi$  set   $\psi:= \varphi \cdot (\theta\circ G)$. Then $\psi \in  W^{1,q}(X, \mu)$ for every $q$. By Corollary \ref{partiW1p} we may apply  \eqref{partitraccia}, \eqref{partitraccia2} to 
$\psi$, with $q=p$, obtaining respectively
\begin{equation}
\label{partitraccian}
  \int_{\{G=0\}} | \varphi|^p |D_HG|_H  \,d\rho  =  p   \int_{G^{-1}(-\delta, 0)}  |\psi|^{p-2}\psi \langle D_H\psi, D_HG\rangle_H   \,d\mu +  \int_{G^{-1}(-\delta, 0)} LG \,|\psi|^p \,d\mu , 
  \end{equation}
and 
\begin{equation}
\label{partitraccia2n}
 \int_{\{G=0\}} | \varphi|^p   \,d\rho = \ds p \int_{G^{-1}(-\delta, 0)}  |\psi|^{p-2}\psi\frac{\langle D_H\psi, D_HG\rangle_H }{|D_HG|_H }  \,d\mu + \int_{G^{-1}(-\delta, 0)} \di \bigg(\frac{D_HG}{|D_HG|_H}\bigg)  \,| \psi|^p \,d\mu .
\end{equation}
If \eqref{strana1bis} holds, we estimate the right hand side of   \eqref{partitraccian}, while if  \eqref{stranabis} holds we estimate the right hand side of \eqref{partitraccia2n}. In both cases we get \eqref{maggbuona}. 
 \end{proof}
 
Below we state some properties of traces. 
To start with, we prove  a version of the integration by parts formula. 

\begin{Proposition}
\label{Prop:particlassica}
Let  $\varphi\in W^{1,p}(\O, \mu)$, $\psi\in W^{1,q}(\O, \mu)$ with $pq/(p+q)>1$. Then
\begin{equation}
\label{particlassica}
\int_{\O} D_k\varphi \,\psi \,d\mu = -\int_{\O} D_k\psi \,\varphi \,d\mu +  \int_{\O} \hat{v}_k\varphi \,\psi\, d\mu + \int_{G^{-1}(0)} \frac{D_k G}{|D_HG|_H }\Tr \varphi \,\Tr \psi \,d\rho , \quad k\in \N .
\end{equation}
\end{Proposition}
\begin{proof} 
By Lemma \ref{product}, $\varphi\,\psi\in W^{1,r}(\O, \mu)$ for every $r\in (1, pq/(p+q))$. Formula  \eqref{parti} applied to $\varphi\,\psi$ yields
$$\int_{\O} D_k\varphi \,\psi \,d\mu = -\int_{\O} D_k\psi \,\varphi \,d\mu +  \int_{\O} \hat{v}_k\varphi \,\psi\, d\mu + \int_{G^{-1}(0)} \frac{D_k G}{|D_HG|_H }\Tr( \varphi \, \psi) d\rho , \quad k\in \N .$$
It remains to show that the trace of $ \varphi  \psi$ at $G^{-1}(0)$ coincides with the product of the respective traces. 
This follows as in Lemma \ref{product}, choosing sequences  $(\varphi_n)$, $(\psi_n)$  of smooth cylindrical functions whose restrictions to $\O$ converge to $\varphi$, $\psi$, in $ W^{1,p}(\O, \mu)$, $W^{1,q}(\O, \mu)$ respectively, so that $(\varphi_n \psi_n)$  converges to $\varphi \psi$, in $ W^{1,r}(\O, \mu)$ for $r<pq/(p+q)$. 
By estimate \eqref{partitraccia2} with $q=1$, $(\varphi_n\psi_{n|G^{-1}(0)})$ converges to $\Tr( \varphi  \psi) $ in $L^1(G^{-1}(0), \rho)$. On the other hand, still by estimate \eqref{partitraccia2}, 
$(\varphi_{n|G^{-1}(0)})$ converges to $\Tr( \varphi ) $ in $L^s(G^{-1}(0), \rho)$ for every $s< p$, 
$(\psi_{n|G^{-1}(0)})$ converges to $\Tr( \psi ) $ in $L^s(G^{-1}(0), \rho)$ for every $s< q$. By the H\"older inequality, 
$(\varphi_n\psi_{n|G^{-1}(0)})$ converges to $\Tr( \varphi ) \Tr( \psi ) $ in $L^1(G^{-1}(0), \rho)$, and the statement follows. 
\end{proof}

\begin{Proposition}
\label{tracciapuntuale}
For every  $\varphi\in W^{1,p}(X, \mu)$, the trace of $\varphi_{|\O}$ at $G^{-1}(0)$ coincides $\rho$-a.e. with the restriction to $G^{-1}(0)$ of any precise version $\widetilde{\varphi}$ of $\varphi$. 

As a consequence, the traces at $G^{-1}(0)$  of  $\varphi_{|\O} = \varphi_{|G^{-1}(-\infty, 0)}$ and of $\varphi_{|G^{-1}(0, +\infty)}$ coincide. 
\end{Proposition}
\begin{proof}
For any sequence of Lipschitz continuous functions $(\varphi_n)$ converging to $\varphi_{|\O}$ in $W^{1,p}(\O, \mu)$, we have by \eqref{partitraccia} with $q=1$
$$\int_{G^{-1}(0)} |\widetilde{\varphi}-\varphi_n|\, |D_HG|_H\,d\rho = \int_{\O}(\sign (\varphi - \varphi_n)
\langle D_H(\varphi - \varphi_n), D_HG\rangle + LG |\varphi - \varphi_n|) d\mu .$$
Letting $n\to \infty$ the left hand side converges to $\int_{G^{-1}(0)} |\widetilde{\varphi}-\Tr \varphi|\, |D_HG|_H\,d\rho$, while the right hand side vanishes. Hence, $\widetilde{\varphi}-\Tr \varphi =0$ $\rho$-a.e.

Note that replacing $G$ by $-G$, the spaces $W^{1,p} (G^{-1}(0, +\infty))$ are well defined for every $p>1$. The trace of $\varphi_{|\O}$ $=$  $\varphi_{|G^{-1}(-\infty, 0)}$  and of $\varphi_{|G^{-1}(0, +\infty)}$ coincide, since both of them 
are $\rho$-a.e. equal to the restriction to $G^{-1}(0)$ of any precise version $\widetilde{\varphi}$ of $\varphi$. 
\end{proof}

\begin{Remark}
Formulae \eqref{parti}, \eqref{partitraccia}, \eqref{partitraccia2} may be taken as starting points to show other formulae and properties. For instance, 
\begin{itemize}
\item[(i)] taking $\varphi \equiv 1 $, formula \eqref{partitraccia2} shows that $\rho(G^{-1}(0)) <+\infty$ and gives a  way to compute or estimate it; 
\item[(ii)] taking any $h\in H$, for every $\varphi \in W^{1,p}(X, \mu)$ with $p>1$ we get
\begin{equation}
\label{partial_h}
\int_{\O} (\partial_h \varphi - \hat{h}\varphi)d\mu = \int_{G^{-1}(0)} \Tr \varphi \, \frac{\partial_h G}{|D_HG|_H}\,d\rho, 
\end{equation}
where  $\partial_h \varphi = \langle D_H\varphi, h\rangle_H$; 
\item[(iii)] taking any vector field $\Phi \in W^{1,p}(X, \mu; H)$ with $p>1$, $\Phi(x) = \sum_{k\in \N}\varphi_k(x) v_k$, applying \eqref{parti} to each $\varphi_k$ and summing up we obtain a version of the Divergence Theorem, 
$$\int_{\O} \d\,\Phi\, d\mu = \int_{G^{-1}(0)} \langle \Tr \Phi, \frac{D_HG}{|D_HG|_H}\rangle \,d\rho ,$$
where $ \Tr \Phi =  \sum_{k\in \N}(\Tr \varphi_k) v_k$. 
\end{itemize}
\end{Remark}

While the range of the trace operator is difficult to characterize, its kernel may be described in a simple way.

\begin{Proposition}
\label{equivalenza}
For every $p>1$ the kernel of the trace operator in $W^{1,p}(\O, \mu)$ consists of the elements $\varphi \in W^{1,p}(\O, \mu)$ whose null extension $\varphi_0$  to the whole $X$ belongs to $W^{1,p}(X, \mu)$.
\end{Proposition}
\begin{proof}
For every $\varphi\in W^{1,p}(\O, \mu)$ denote by $\varphi_0$ the null extension of $\varphi$: $\varphi_{0|\O} = \varphi$, 
$\varphi_{0|X\setminus \O} = 0$. 

Let $\varphi\in W^{1,p}(\O, \mu)$ have null trace at $G^{-1}(0)$. Then for every smooth cylindrical function $\psi:X\mapsto \R$ and for every $h\in H$, applying \eqref{partial_h} to the product $\psi \varphi$ we get 
$$ \int_{\O}\partial_h \psi \, \varphi\,d\mu 
= \int_{\O} (- \psi \partial_h \varphi + \psi \varphi \hat{h}) d\mu + \int_{G^{-1}(0)}\psi  \Tr \varphi \,\frac{\partial_h G}{|D_HG|_H}\,d\rho$$
so that, since the last integral vanishes, 
$$\int_X \partial_h \psi \,\varphi_0\,d\mu = \int_{\O}\partial_h \psi \, \varphi\,d\mu =
\int_{\O} (- \psi \partial_h \varphi+ \psi  \varphi \hat{h})d\mu 
= \int_{X} (- \psi (\partial_h \varphi)_0+ \psi  \varphi_0 \hat{h})d\mu .$$
Therefore, the function $(\partial_h \varphi)_0\in L^p(X, \mu)$ is the generalized partial derivative of $\varphi$ along $h$. Hence, the generalized derivative of $\varphi_0$, in the sense of \cite[Def. 5.2.9]{Boga} is  $ (D_H\varphi)_0$, namely the null extension of $D_H\varphi $ to $X$. 
By \cite[Cor. 5.4.7]{Boga}, $\varphi_0 \in W^{1,p}(X,  \mu)$. 

Let now $\varphi$ be such that $\varphi_0\in W^{1,p}(X,  \mu)$. We use an argument from \cite{DPL}: replacing $G$ by $-G$, and using \eqref{partitraccia} or \eqref{partitraccia2} with $\O$ replaced by $G^{-1}(0, \infty)$, we see that the trace at $G^{-1}(0)$ of $\varphi_{0|G^{-1}(0, +\infty)}$ vanishes. On the other hand, since $\varphi_0 \in W^{1,p}(X,  \mu)$, the traces at $G^{-1}(0)$ of  $ \varphi_{0| \O}$ and of $ \varphi_{0|G^{-1}(0, +\infty)}$ coincide by Proposition \ref{tracciapuntuale}. Since $ \varphi_{0|  \O}=\varphi$, 
then $\Tr \varphi =0$. 
\end{proof}

The space $\oo{W}^{1,2} (\O,\mu)$, consisting of   (classes of equivalence of) functions $\varphi:\O \mapsto \R$ whose null extension to the whole $X$ belongs to $W^{1,2}(X, \mu)$ was considered in  \cite{DPL,DPL2}. 
Proposition \ref{equivalenza} shows that under our assumptions such a space is just the kernel of the trace operator in 
$W^{1,2}(\O, \mu)$. 
 
\begin{Definition}
\label{W1p0}
For $p>1$ we set
$$\oo{W}^{1,p} (\O,\mu) = \{\varphi \in   W^{1,p}(\O, \mu):\;\Tr \varphi =0\} = \{ \varphi\in W^{1,p}(\O, \mu):\;\varphi_0\in W^{1,p}(X,  \mu)\}. $$
\end{Definition}

We end this section rewriting some consequences of the above results in terms of the papers \cite{Hino,Diego} (see also the related papers \cite{Diego0,FH}). 

The function $\one_{\O}$ is of bounded variation in the sense of  \cite{Hino,Diego}, since for every $h\in H$ and for every cylindrical $\varphi\in C^1_b(X)$ we have, by \eqref{partial_h},
$$\int_{\O} (\partial_h\varphi - \varphi \hat{h})d\mu = \int_{G^{-1}(0)} \varphi\frac{\partial_hG}{|D_HG|_H}\,d\rho ,$$
and the right hand side may be rewritten as $\int_X \varphi \,d\langle \nu, h\rangle$, where 
\begin{equation}
\label{nu}
d\nu := \one_{G^{-1}(0)} \frac{D_HG}{|D_HG|_H}\, d\rho
\end{equation}
is a $H$-valued measure with finite total variation. Therefore, the perimeter of $\O$ is finite and it coincides with $\rho(G^{-1}(0))$. More generally, on $ {G^{-1}(0)}$ the measure  $\rho$ coincides with  the perimeter measure   of \cite{Hino,Diego}, and $\sigma (x) =  D_HG(x)/|D_HG(x)|_H$ is the $H$-valued unit vector field in the polar decomposition of $\nu$.  
Since the coarea formula holds for the perimeter measure (\cite[Thm. 3.7]{Diego0}), arguing as in Lemma \ref{Cor:coarea} and Proposition \ref{W11} one can see that  $\sigma (x) =  D_HG(x)/|D_HG(x)|_H$ at the level sets $G^{-1}(\xi)$, for almost all $\xi\in (-\delta, \delta)$. The fact that it holds precisely for  $\xi =0$ does not follow directly from the above mentioned papers.

%
%

\section{Examples}
\label{Examples}

\subsection{Halfspaces}
\label{half}

If $\O$ is a halfspace of the type $\{x\in X: \; \hat{h}(x) > 0\}$  for some $\hat{h}\in X^*$, we can characterize the set of the traces of the elements of  $W^{1,p}(\O, \mu)$ at the boundary $\partial \O= \{x\in X: \; \hat{h}(x) =0\}$. 

To understand what is going on, we recall briefly a classical result in the case where $X=\R^n$, the Gaussian measure is replaced by the Lebesgue measure $dx$, $\O = \{x\in \R^n:\; \langle e, x\rangle_{\R^n}>0\}$ and $e$ is a unit vector. 
Then, splitting $\R^n =$  span$\,e \oplus e^{\perp} = \R \oplus \R^{n-1}$ and identifying 
$\partial \O = \{0\} \times \R^{n-1}$ with $ \R^{n-1}$,  the space of the traces at $\partial \O $ of the elements of $W^{1,p}(\O)$ with $p>1$ is precisely the fractional Sobolev space $W^{1-1/p,p}(\R^{n-1})$. This result may be proved in several ways, and it is the first step to characterize the spaces of the traces of the Sobolev functions at the boundaries of other regular sets as fractional Sobolev spaces. 

The most popular proof uses interpolation, through the characterization of real interpolation spaces as trace spaces. Indeed,  for every couple of Banach spaces $E$, $F$ such that  $F\subset E$ with continuous embedding, and for each  $\theta \in (0, 1)$, $p\geq 1$, the real  interpolation space 
$(E, F)_{1-\theta, p}$ coincides with the set of the traces at  $t=0$ of the elements of $V(p, \theta, E,F)$ defined as the space of the functions 
$V\in W^{1,p}_{loc}((0, +\infty);E) \cap L^{p}_{loc}((0, +\infty);F)$ such that 
$$t\mapsto t^{\theta -1/p} V(t) \in L^{p} ((0, +\infty),  dt;F), \quad t\mapsto t^{\theta -1/p} V'(t) \in L^{p} ((0, +\infty), dt;E), $$
moreover the norm of $(E, F)_{1-\theta, p}$ is equivalent to 
$$ |a|_{\theta, p}  := \inf\bigg\{   \bigg(\ds \int_{0}^{+\infty} t ^{\theta p-1 } (\| V(t) \|_{F}^p +  \| V'(t) \|_{E}^p) \,dt\bigg)^{1/p} :\; V\in V(p, \theta,  E,F), \;V(0)=a\bigg\} .
$$
See e.g. \cite[\S 1.8.2]{Tr}. Then, taking $\theta= 1/p$, $E=L^p(\R^{n-1})$, $F= W^{1,p}(\R^{n-1})$  one checks that $V\in  V(p, 1/p, L^p(\R^{n-1}),W^{1,p}(\R^{n-1}))$  iff $(t,x)\mapsto V(t)(x)$ belongs to 
$W^{1,p}(\O)$, and this implies that the space of the traces at $\partial \O$ of the elements of $W^{1,p}(\O)$ coincides with the interpolation space $(L^p(\R^{n-1}), W^{1,p}(\R^{n-1}))_{1-1/p, p}$ that in its turn is known to coincide with $W^{1-1/p,p}(\R^{n-1})$. 

We shall follow this approach also for infinite dimensional halfspaces. 

Let  $\O$ be the halfspace  $\{x\in X: \; \hat{h}(x) > 0\}$  for some $\hat{h}\in X^*$. We set  $h := Q(\hat{h})$ and without loss of generality we assume that   $|h|_H=1$, so that $\hat{h}(h)=1$. 

Denoting by $\Pi_h(x) = \hat{h}(x)h$, as
in \S \ref{Surface measures} we split $X = \Pi_h(X) \oplus (I-\Pi_h)(X)$ $=$  span$\,h \oplus Y$,  
 where the linear span of $h$ plays the role of $F$ and $Y$ plays the role of $\widetilde{F}$. 
The image measure  $\mu \circ \Pi_h ^{-1}$ is the standard Gaussian measure in $\R$, identified with  span$\,h$, and the image measure $\mu \circ (I-\Pi _h)^{-1}$  is a centered nondegenerate Gaussian measure $\mu_Y$ in $Y$. The Cameron--Martin space is also decomposed as $H= $  span$\,h \oplus H_Y$,  where $H_Y$ is the orthogonal space to $h$ in $H$.  

The spaces $L^{p}(X, \mu)$ and  $W^{1,p}(X, \mu)$ are identified with $L^{p}(\R \times Y,N_{0,1}(dt) \otimes  \mu_Y)$ and $W^{1,p}(\R \times Y,N_{0,1}(dt) \otimes  \mu_Y)$, respectively.

 \begin{Proposition}
For $p>1$ the space of the traces on $\partial \O = \{x\in X: \; \hat{h}(x) = 0\}$ of the elements of $W^{1,p}(\O, \mu)$ coincides with the real  interpolation space $(L^{p}(Y,  \mu_Y), W^{1,p}(Y,  \mu_Y))_{1-1/p, p}$. 
\end{Proposition}
\begin{proof} The proof is the same as the above  proof for the Lebesgue measure in finite dimensions through the following lemmas.  \end{proof}

\begin{Lemma}
\label{tracceGauss}
For every couple of Banach spaces $E$, $F$ such that  $F\subset E$ with continuous embedding, and for each  $\theta \in (0, 1)$, $p\geq 1$,  the real  interpolation space 
$(E, F)_{1-\theta, p}$ coincides with the set of the traces at  $t=0$ of the functions 
in $W(p, \theta, E,F)$ defined as the space of the functions 
$V\in W^{1,p}_{loc}((0, +\infty);E) \cap L^{p}_{loc}((0, +\infty);F)$ such that 
$$t\mapsto t^{\theta -1/p} V(t) \in L^{p} ((0, +\infty), N_{0,1}(dt);F), \quad t\mapsto t^{\theta -1/p} V'(t) \in L^{p} ((0, +\infty), N_{0,1}(dt);E). $$
Moreover the norm of $(E, F)_{\theta, p}$ is equivalent to 
$$ 
[a]_{\theta, p}  := \inf\bigg\{   \bigg(\ds \int_{0}^{+\infty} t ^{\theta p-1 } (\| V(t) \|_{F}^p+  \| V'(t) \|_{E}^p)N_{0,1}(dt)\bigg)^{1/p} :\; V\in W(p, \theta, E,F), \;V(0)=a\bigg\}. 
$$
\end{Lemma}
\begin{proof} Since $V(p, \theta, E,F)\subset W(p, \theta, E,F)$ then $(E, F)_{1-\theta, p}$ is contained in the set of the traces at $0$ of the elements of $W(p, \theta, E,F)$, and $[a]_{\theta, p} \leq  |a|_{\theta, p}$, for each $a\in (E, F)_{1-\theta, p}$. Conversely, let $\eta \in C^{\infty}(\R)$ be such that $\eta \equiv 1$ in $[0, 1]$, $\eta \equiv 0$ in $[2, +\infty)$. For every $V\in W(p, \theta, E,F)$ set $\widetilde{V}(t) = \eta (t)V(t)$. Then $\widetilde{V}\in V(p, \theta, E,F)$, and 
$\|\widetilde{V}\|_{V(p, \theta, E,F)} \leq C\|V\|_{W(p, \theta, E,F)}$ with $C$ independent of $V$. Since $V$ and $W$ coincide a.e. on $(0, 1)$ they have the same trace at $0$, and the statement follows. 
\end{proof}

\begin{Lemma}
\label{Le:iso}
For every $v\in W^{1,p}(\O, \mu)$ set $V(t)(y) =  v(th + y)$, for $t\geq 0$ and $y\in Y$. Then the mapping 
$$ W^{1,p}(\O, \mu)\mapsto W(p, 1/p,  L^p(Y, \mu_Y),W^{1,p}(Y, \mu_Y)), 
\quad v\mapsto V $$
is an isomorphism. 
\end{Lemma}
\begin{proof}
By Lemma \ref{tracceGauss}, for $\theta =1/p$ the space $W(p, 1/p,  L^p(Y, \mu_Y),W^{1,p}(Y, \mu_Y))$ coincides with 
$L^p((0, +\infty),  N_{0,1}(dt);W^{1,p}(Y, \mu_Y))$ $\cap$ $W^{1,p}  ((0, +\infty), N_{0,1}(dt);L^p(Y, \mu_Y))$, and we have
\begin{equation}
\label{norma}
\|V\|_{W(p, 1/p,  L^p(Y, \mu_Y),W^{1,p}(Y, \mu_Y))}^p = \int_0^{+\infty} ( \|V(t)\|^p_{W^{1,p}(Y, \mu_Y)} + \|V'(t)\|^p_{L^{p}(Y, \mu_Y)} )N_{0,1}(dt). 
\end{equation} 
Let $v\in  W^{1,p}(\O, \mu)$. Then for every $x\in \O$, $x = th + y$ with $t>0$ and $y\in Y$ we have $D_Hv(th+y) = h \, \partial/ \partial h\, v(th +y)  + D_{H_Y}v(th +y)$. Moreover  for every $t>0$, $V(t)\in W^{1,p}(Y, \mu_Y)$, and  $D_{H_Y}V(t)(y) =  D_{H}^Yv(th +y)$. To show that $V\in W(p, 1/p,  L^p(Y, \mu_Y),W^{1,p}(Y, \mu_Y))$ (in particular, to show that it is measurable with values in $W^{1,p}(Y, \mu_Y)$) we approach $v$ by the restrictions to $\O$ of a sequence of smooth cylindrical functions $v_n$ (see Proposition \ref{approx}). 
The corresponding functions $V_n$ defined by $V_n(t)(y) =  v_n(th + y)$ belong to 
$C^1_b([0, +\infty); C^1_b(Y))$ $\subset $ $W(p, 1/p,  L^p(Y, \mu_Y),W^{1,p}(Y, \mu_Y))$, and $V_n'(t)(y) = \partial v_n/\partial h(th+y)$. Without loss of generality we may assume that $(v_n)$ and $(D_Hv_n)$  converge pointwise to $v$ and to $D_Hv$ a.e. in $\O$, respectively. 
Moreover $(V_n)$ is a Cauchy sequence in $W(p, 1/p,  L^p(Y, \mu_Y),W^{1,p}(Y, \mu_Y))$, since by \eqref{norma}
we have
$$\begin{array}{l}
\ds \|V_n-V_m\|_{W(p, 1/p,  L^p(Y, \mu_Y),W^{1,p}(Y, \mu_Y))}^p 
\\
\\
\ds = \int_0^{+\infty} \bigg( \int_Y (|(v_n-v_m)(th+y)|^p + |(D_{H}^{Y}v_n - D_{H}^{Y}v_m)(th +y)|_{H}^p) d\mu_Y \bigg) N_{0,1}(dt)
\\
\\ 
+\ds
\int_0^{+\infty}  \bigg( \int_Y    |\frac{\partial}{\partial h} v_n(th +y)- \frac{\partial}{\partial h} v_m(th +y)|^p 
d\mu_Y \bigg) N_{0,1}(dt)
\\
\\
\ds \leq \int_{\O} (|(v_n-v_m)(x)|^p + |(D_{H} v_n-D_{H} v_m) (x)|_{H}^p  ) d\mu . 
\end{array}$$
Therefore, the pointwise limit function $V$ belongs to  
$W(p, 1/p, L^p(Y, \mu_Y),W^{1,p}(Y, \mu_Y))$ and 
$$\|V\|_{W(p, 1/p,  L^p(Y, \mu_Y),W^{1,p}(Y, \mu_Y))} \leq  \|v\|_{W^{1,p}(\O, \mu)}. $$
(The norms are  equal  for $p=2$).

Conversely, let $V\in W(p, 1/p,  L^p(Y, \mu_Y),W^{1,p}(Y, \mu_Y))$ $=$  $L^p((0, +\infty),  N_{0,1}(dt);W^{1,p}(Y, \mu_Y))$ $\cap$ $W^{1,p}  ((0, +\infty), N_{0,1}(dt);L^p(Y, \mu_Y))$. We extend $V$ to the whole $\R $ by reflection, setting
$$\widetilde{V}(t) = V(t), \;t\geq 0; \quad \widetilde{V}(t) = V(-t), \;t<0.$$
The extension $\widetilde{V}$ belongs to $L^p(\R,  N_{0,1}(dt);W^{1,p}(Y, \mu_Y))$ $\cap$ $W^{1,p}  (\R, N_{0,1}(dt);L^p(Y, \mu_Y))$, and   $\widetilde{V}'(t) = V'(t)$ for a.e. $t>0$, $\widetilde{V}'(t) = -V'(-t)$ for a.e. $t<0$. 

Recalling the identification of $W^{1,p}(X, \mu)$ with 
$W^{1,p}(\R \times Y,N_{0,1}(dt) \otimes  \mu_Y)$, our aim is to show that the function $v$ defined by 
$$v(t,y) = V(t)(y), \quad t\in \R, \;y\in Y$$
belongs to $W^{1,p}(\R \times Y,N_{0,1}(dt) \otimes  \mu_Y)$, and that $\|v\|_{W^{1,p}}\leq C\|V\|_{W(p, 1/p,  L^p(Y, \mu_Y),W^{1,p}(Y, \mu_Y))}$. Indeed in this case, the restriction of $v$ to $(0, +\infty)\times Y$ belongs to the space 
$W^{1,p}((0, +\infty)\times Y,$ $N_{0,1}(dt) \otimes  \mu_Y)$, with the same estimate of its norm. 

By a tedious but standard procedure, namely approaching $V$ by a sequence of measurable simple functions 
$V_n(t) = \sum_{k=1}^{k_n} \one_{A_{k,n}}(t) \varphi_{k,n}$ in $L^p((0, +\infty); N_{0,1}(dt)); L^p(Y, \mu_Y))$ and then 
approaching each $ \varphi_{k,n}$ by simple $\mu_Y$-measurable functions with real values in $L^p(Y, \mu_Y)$, one can see that $v$ is measurable in $\R \times Y$ with respect to $N_{0,1}(dt) \otimes  \mu_Y$. Integrating we get 
$$\int_ {\R} \int_Y |v(t,y)|^p  \mu_Y(dy) \,N_{0,1}(dt) = \int_{\R}\|\widetilde{V}(t )\|^p_{L^p(Y, \mu_Y)}   \,N_{0,1}(dt),$$
so that $v\in L^p(\R \times Y,N_{0,1}(dt) \otimes  \mu_Y)$. To prove that it belongs to $W^{1,p}(\R \times Y,N_{0,1}(dt) \otimes  \mu_Y)$, let us remark that for every smooth cylindrical function $\varphi: \R \times Y\mapsto \R$ the function $\theta :\R\mapsto \R$ defined by
$$\theta (t) = e^{-t^2/2} \int_Y \varphi_t(t,y)\widetilde{V}(t )(y)\,  \mu_Y(dy) =  e^{-t^2/2} \langle \varphi_t(t,\cdot )\widetilde{V}\rangle _{L^{p'}(Y, \mu_Y), L^{p}(Y, \mu_Y)}$$
is weakly differentiable in $\R$, with weak derivative 
$$\theta'(t) = -t \theta(t) + e^{-t^2/2}\int_Y (\varphi_t(t,y)\widetilde{V}(t )(y)+ \varphi(t,y)\widetilde{V}'(t )(y)) \mu_Y(dy) .$$
Integrating over $\R$ we get 
$$\int_{\R}\int_Y \varphi_t(t,y) v(t,y)\,  \mu_Y(dy) \,N_{0,1}(dt) = \int_{\R}\int_Y\varphi_t(t,y)\widetilde{V}(t )(y)\,  \mu_Y(dy) \,N_{0,1}(dt) $$
$$ =  -\int_{\R} \int_Y \varphi(t,y)\widetilde{V}'(t )(y)\,  \mu_Y(dy) \,N_{0,1}(dt) + \int_{\R} \int_Y t\varphi(t,y)\widetilde{V}(t )(y)\,  \mu_Y(dy) \,N_{0,1}(dt),  $$
namely the function $(t,y) \mapsto  \widetilde{V}'(t )(y)$ is the weak derivative of $v$ in the direction of $(1,0)$ in the sense of 
Sect. \ref{Notation and preliminaries}. It  belongs to $L^p(\R \times Y,N_{0,1}(dt) \otimes  \mu_Y)$ since, as before, 
$$\int_ {\R} \int_Y |v_t(t,y)|^p  \mu_Y(dy) \,N_{0,1}(dt) = \int_{\R}\|\widetilde{V}'(t )\|^p_{L^p(Y, \mu_Y)}   \,N_{0,1}(dt).$$
Moreover, $\widetilde{V} (t )\in W^{1,p}(Y, \mu_Y)$ for a.e. $t\in \R$, and therefore for every $k\in H\cap Y$ and for each smooth cylindrical function $\varphi$ we have
$$\int_Y \varphi_k(t,y)\widetilde{V} (t )(y)\,\mu_Y(dy) =$$
$$= -\int_Y \varphi(t,y)   \langle D_{H_Y}\widetilde{V}(t )(y), k\rangle_H\,\mu_Y(dy) + \int_Y \hat{k}(y)\varphi(t,y)   \varphi_k(t,y)\widetilde{V} (t )(y)\,\mu_Y(dy). $$
Integrating over $\R$ with respect to $N_{0,1}(dt)$, we obtain that 
$v$ is weakly differentiable in any direction $(0,k)  \in H\cap Y$, with $\partial v/\partial k(t,y) = \langle D_{H_Y}\widetilde{V}(t )(y), k\rangle_H$. Hence, $v$ is weakly differentiable in any direction $k\in H$, and the weak gradient at any $(t,y)$ is given by $(\widetilde{V}'(t )(y),  D_{H_Y}\widetilde{V}(t )(y))$. 
By \cite[\S 5.2, \S 5.4]{Boga}, $v\in W^{1,p}(\R \times Y,N_{0,1}(dt) \otimes  \mu_Y)$, the weak gradient at $(t,y)$ coincides with $D_Hv(t,y)$, and 
$$\int_ {\R} \int_Y|D_Hv(t,y)|_H^p  \mu_Y(dy) \,N_{0,1}(dt) \leq 2^{p-1}\int_{\R}\int_Y(|\widetilde{V}'(t )(y)|^p  + |D_{H_Y}\widetilde{V}(t )(y)|_H^p ) \mu_Y(dy) \,N_{0,1}(dt) = $$
$$= 2^{p-1}(\|\widetilde{V}'\|_{L^p(\R , N_{0,1}(dt); L^p(Y, \mu_Y)) }^p+ \|\widetilde{V}\|_{L^p(\R , N_{0,1}(dt);  W^{1,p}(Y, \mu_Y)) }^p).$$
\end{proof}

Let us set 
$$T_p := (L^{p}(Y,  \mu_Y), W^{1,p}(Y,  \mu_Y))_{1-1/p, p}. $$
The trace space $T_p$ may be characterized using   the  Ornstein--Uhlenbeck operator on $Y$, $L_Y: W^{2,p}(Y, \mu_Y)\mapsto L^{p}(Y, \mu_Y)$. It is defined by 
$$L_Yu = \di(D_{H_Y}u), $$
where $\di$ denotes the Gaussian divergence with respect to $\mu_Y$ (e.g., \cite[\S 5.8]{Boga}). 
By the Meyer inequalities (\cite[\S 5.6, 5.7]{Boga}),  $W^{1,p}(Y,  \mu_Y)$ is the domain of $(I-L_Y)^{1/2}$. Using  the Reiteration Theorem (\cite[\S 1.10.2]{Tr}) yields
$$
T_p= (L^{p}(Y,  \mu_Y), D(I-L_Y))_{1/2-1/2p, p}= (L^{p}(Y,  \mu_Y), W^{2,p}(Y,  \mu_Y))_{1/2-1/2p, p}.$$
Since $L_Y$ is a $m$-dissipative operator that generates an analytic semigroup, all the classical characterizations of the real interpolation spaces between its domain and the underlying space $L^{p}(Y, \mu_Y)$ hold, see e.g. \cite[\S 1.13, 1.14]{Tr}.  In particular, $T_p$ is the set of all $f\in L^{p}(Y,  \mu_Y)$
such that 
\begin{equation}
\label{interp1}
\int_0^{\infty}t^{-(p+1)/2 }\|e^{tL_Y}f-f\|_{L^{p}(Y,  \mu_Y)}^pdt <\infty ,
\end{equation}
or, equivalently, such that 
\begin{equation}
\label{interp2}
\int_0^{\infty}t^{ (p-1)/2 }\|L_Ye^{tL_Y}f\|_{L^{p}(Y,  \mu_Y)}^pdt <\infty ,
\end{equation}
or, equivalently, such that 
\begin{equation}
\label{interp3}
\int_0^{\infty}\lambda ^{ (p-3)/2 }\|L_Y(\lambda I- L_Y)^{-1}f\|_{L^{p}(Y,  \mu_Y)}^p d\lambda  <\infty .
\end{equation}
For $p=2$, $L_Y$ is self-adjoint in $L^{2}(Y,  \mu_Y)$ and we have in addition the characterization through the spectral decomposition, 
\begin{equation}
\label{interp4}
(L^{2}(Y,  \mu_Y), W^{1,2}(Y,  \mu_Y) )_{1/2, 2}
= D((-L_Y)^{1/4})  = \{ f\in L^{2}(Y,  \mu_Y): \; \sum_{k=1}^{\infty}  k^{1/2}\|I_k(f)\|^2 <\infty\},
\end{equation}
where  $I_k(f)$ is the orthogonal projection on the subspace of $L^{2}(Y,  \mu_Y)$ generated by the Hermite polynomials of order $k$. See \cite[p. 78, p. 215]{Boga}. They are the Hermite polynomials  $P$ in $Y$ such that $ L_YP = -k P$.

\vspace{3mm}

The integration formulae of  Section \ref{sect:traces} are particularly simple in this case. Indeed, $G(x) = -\hat{h}(x)$, $D_HG(x) = -h$ (constant), so that $|D_HG(x)|_H =1$, $LG  =  \sum_{j} \langle v_j, h\rangle_H\hat{v_j}(x)$  so that  for each $k$ \eqref{parti} becomes
\begin{equation}
\label{partisemispazio}
\int_{\O} D_k\varphi \,d\mu =   \int_{\O} \hat{v_k}\varphi \,d\mu -  \langle v_k, h\rangle_H \int_{Y}  \varphi \,d \mu_Y ,  
\end{equation}
and formulae \eqref{partitraccia}, \eqref{partitraccia2} become
\begin{equation}
\label{tracciasemispazio}
\int_{Y} | \varphi|^p   \,d \mu_Y  =  -  p  \int_{\O}  | \varphi|^{p-2}\varphi \, \langle   h, D_H\varphi \rangle \,d\mu +  \int_{\O} \sum_{j}  \langle v_j, h\rangle_H\hat{v_j}(x)\,| \varphi|^p \,d\mu . 
\end{equation}
Note that  if  $\langle v_k, h\rangle_H = \hat{h}(v_k)= 0$, then $v_k\in \partial \O$, $D_k$ may be considered as a tangential derivative, and formula \eqref{partisemispazio} gets similar  to the standard integration formula \eqref{partiX}. 

\vspace{3mm}

Using the Ornstein--Uhlenbeck semigroup and the characterization  \eqref{interp2} it is possible to define a nice extension operator ${\mathcal E }: T_p\mapsto W^{1,p}(\O, \mu)$, just setting
 $$({\mathcal E }f)(th+y) := (e^{t^2L_Y}f)(y), \quad t>0, \;y\in Y.$$

\begin{Proposition}
The operator ${\mathcal E }$ is bounded from $T_p$ to $W^{1,p}(\O, \mu)$. 
\end{Proposition}
\begin{proof}
By Lemma \ref{Le:iso} it is sufficient to prove that for every $f\in T_p$, the function $V(t) = e^{t^2L_Y}f$ belongs to the space $W(p, 1/p,  L^p(Y, \mu_Y),W^{1,p}(Y, \mu_Y))$ $=$
$L^p((0, +\infty),  N_{0,1}(dt);W^{1,p}(Y, \mu_Y))$ $\cap$ $W^{1,p}  ((0, +\infty), N_{0,1}(dt);L^p(Y, \mu_Y))$, with norm estimated by $c\|f\|_{T_p}$, $c$ independent of $f$. 

Since $\|e^{t^2L_Y}f\|_{L^p(Y, \mu_Y)} \leq \|f\|_{L^p(Y, \mu_Y)}$ for every $t > 0$, then $V\in L^{p}  ((0, +\infty), N_{0,1}(dt);$ $ L^p(Y, \mu_Y))$. Moreover,  
$V\in C^{\infty}((0, +\infty);L^p(Y, \mu_Y))$ and $V'(t) = 2t L_Ye^{t^2L_Y}f$, hence estimate \eqref{interp2} shows that $V'\in L^p((0, +\infty), N(0,1)(dt); L^p(Y, \mu_Y))$ through the obvious change of variable $t^2=s$. 

To show that $V\in L^p((0, +\infty), N(0,1)(dt); W^{1,p}(Y, \mu_Y))$ through the same change of variables we need to know that $s\mapsto s^{-1/2p}\|e^{sL^Y}f\|_{W^{1,p}(Y, \mu_Y)}$ belongs to $L^p((0, +\infty), e^{-s/2p}ds)$. Since $f\in (L^{p}(Y,  \mu_Y), D(I-L_Y))_{1/2-1/2p, p}$, this is a consequence of  the equivalence $W^{1,p}(Y, \mu_Y)$ $ = D((I-L_Y)^{1/2})$, through next lemma applied with $\theta = 1/2 - 1/2p$, $\alpha = 1/2$. 
\end{proof}
 
 \begin{Lemma}
 Let $X$ be a Banach space and let $A:D(A)\subset X\mapsto X$ be a linear positive operator, generator of an analytic semigroup $e^{tA}$. For every $0<\theta <\alpha $ and  $x\in (X, D(A))_{\theta, p}$, the function 
 $t\mapsto t^{\alpha -\theta }\|A^{\alpha}e^{tA}x \| $ belongs to $L^p((0, +\infty); dt/t)$, and its $L^p$ norm is bounded by $c\|x\|_{(X, D(A))_{\theta, p}}^p$, for some $c$ independent of $x$. 
 \end{Lemma}
 \begin{proof}
 We recall that $t\mapsto t^{1-\theta} \|Ae^{tA}x\|$, $t\mapsto t^{2-\theta} \|A^2e^{tA}x\|$ belong to 
 $L^p ((0, +\infty), dt/t)$ with norms not exceeding $C\|x\|_{(X, D(A))_{\theta, p}}$. Using the interpolation inequality
 $\|A^{\alpha}y\|  \leq C_{\alpha}\|y\|^{1-\alpha} \|Ay\|^{\alpha} $ for every $y\in D(A)$, we get 
 $$t^{1-\theta +\alpha}\|A^{\alpha +1}e^{tA}x\| \leq C_{\alpha} \|t^{1-\theta}Ae^{tA}x\|^{1-\alpha}  \|t^{2-\theta}A^2e^{tA}x\|^{\alpha} \leq C_{\alpha}' (\|t^{1-\theta}Ae^{tA}x\| + \|t^{2-\theta}A^2e^{tA}x\|).$$
Therefore, the function 
$$t\mapsto t^{1-\theta +\alpha}\|A^{\alpha +1}e^{tA}x\|$$
belongs to $L^p ((0, +\infty), dt/t)$, and its norm does not exceed $M\|x\|_{(X, D(A))_{\theta, p}}$ for some $M$ independent of $x$. Since
$$\|A^{\alpha} e^{tA}x\| = \|\int_{t}^{+\infty} A^{\alpha +1}e^{sA}x\,ds\| \leq \int_{t}^{+\infty}\| A^{\alpha +1}e^{sA}x\|ds, \quad t>0, $$
the statement follows applying the Young's inequality 
$$\int_0^{+\infty} t^{(\alpha -\theta)p}\bigg(\int_t^{+\infty}\varphi(s)\,\frac{ds}{s}\bigg)^p \frac{dt}{t}
\leq \frac{1}{(\alpha -\theta)^{p}} \int_0^{+\infty}s^{(1+\alpha -\theta)p}\varphi(s)^p\,\frac{ds}{s} $$
to the function $\varphi(s) = s \| A^{\alpha +1}e^{sA}x\|$. 
\end{proof}

Extension operators are useful in a number of problems, for instance they may be used to extend parts of the theories about PDEs with homogeneous Dirichlet boundary conditions (\cite{DPL,AVN})  to nonhomogeneous boundary conditions. Concerning traces, ${\mathcal E }$ can be used to define a nice projection
$${\mathcal P}: W^{1,p}(\O, \mu)\mapsto \oo{W}^{1,p} (\O,\mu), \quad {\mathcal P}u = u - {\mathcal E }(\Tr u)$$
that allows to split $W^{1,p}(\O, \mu)$ as the direct sum of $\oo{W}^{1,p} (\O,\mu) $ plus a complemented subspace.
Using our good extension operator ${\mathcal E }$ we prove another property of $\oo{W}^{1,p} (\O,\mu) $, similar to the finite dimensional case. 

\begin{Proposition}
The subspace  of $W^{1,p}(\O,\mu)\cap C(\overline{\O})$ consisting of functions that vanish in a neighborhood of $\partial \O$ is dense in $\oo{W}^{1,p} (\O,\mu)$. 
\end{Proposition}
\begin{proof}
Let $u\in \oo{W}^{1,p} (\O,\mu)$ and let $(u_n)$ be a sequence of Lipschitz continuous functions   that  converges to $u $ in $W^{1,p}(\O, \mu)$. Then the sequence $({\mathcal P}u_n)$ converges to ${\mathcal P}u =u$ in 
$W^{1,p} (\O,\mu)$. Let us prove that ${\mathcal E }(\Tr u_n)$ is continuous in $\overline{\O}$. Indeed, $\Tr u_n = u_{n|\partial \O}$ is Lipschitz continuous, and for each Lipschitz continuous $f:Y\mapsto \R$ we have
$$|{\mathcal E }f(t,y) - {\mathcal E }f(t_0,y_0)| \leq \int_Y |f(e^{-t^2}y + \sqrt{1-e^{-2t^2}}z) - f(e^{-t_0^2}y_0 + \sqrt{1-e^{-2t_0^2}}z| )\mu_y(dz)$$
$$\leq C\int_Y (\|e^{-t^2}y  - e^{-t_0^2}y_0\| +  |\sqrt{1-e^{-2t^2}}-  \sqrt{1-e^{-2t_0^2}}|\, \|z\|)\mu_y(dz)$$
where $C$ is the Lipschitz constant of $f$. Therefore,  ${\mathcal E }f$ is continuous (in fact, locally H\"older continuous) in $\overline{\O}$. 

So, ${\mathcal P}u_n\in W^{1,p}(\O, \mu) \cap C(\overline{\O})$ vanish at $\partial \O$ and approach $u$ in $W^{1,p}(\O, \mu)$. In their turn, let us approach each ${\mathcal P}u_n$ by continuous functions that vanish in a neighborhood of $\partial \O$. To this aim, define a piecewise linear function $\eta$ setting  $\eta(\xi) = 0$ for $|\xi|\leq 1$, $\eta(\xi) = 2\xi -2 \sign \xi$ for $1\leq |\xi|\leq 2$, and $\eta(\xi) = \xi$ for $|\xi|\geq 2$. Set
$$u_{k,n} = \frac{\eta \circ (ku_n)}{k}.$$
As easily  seen, for any $n\in \N$ the sequence $(u_{k,n})$ converges to $u_n$ in $W^{1,p}(\O, \mu)$ as $k\to \infty$. Moreover, $u_{k,n} $ vanishes in the set $\{ x\in \O:\; |u_n(x)|< 1/k\}$, which is a neighborhood of $\partial \O$ in $\O$ because $u_n$ is continuous and vanishes at $\partial \O$. 
\end{proof}

\subsection{Regions below graphs}

Simple generalization of halfspaces  are the regions below graphs of good functions. 

Let us keep the notation and setting of \S \ref{half}. We fix $\hat{h}\in X^*$,  such that $\|\hat{h}\|_{L^2(X, \mu)} =1$, and  we set  $h := Q(\hat{h})$. Then  $|h|_H=1$ and $\hat{h}(h)=1$. We split $X = $  span$\,h \oplus Y$, where $Y=$ $ (I-\Pi_h)(X)$, $\Pi_h(x) = \hat{h}(x)h$. The Gaussian measure $\mu\circ (I-\Pi_h)^{-1}$ on $Y$ is denoted by $\mu_Y$. 

Let $F\in \cap_{p>1}W^{2,p}(Y, \mu_Y)$. Choose any Borel precise version of $F$ and set 
$$G:X\mapsto \R, \quad G(x) = \hat{h}(x) - F( (I-\Pi_h)(x)). $$
Then, $G\in \cap_{p>1}W^{2,p}(X, \mu)$ and $D_HG(x) = h - D_{H_Y}F((I-\Pi_h)(x))$, so that $|D_HG(x)|_H^2 = 1+|D_{H_Y}F(I-\Pi_h)(x)|_{H_Y}^2\geq 1$. Hence, $G$ satisfies Hypothesis \ref{HypG}. The sublevel set $\O=G^{-1}(-\infty, 0)$ is just the region below the graph of $F$. 

%
%

Since $|D_HG|_H \geq 1$ at $G^{-1}(0)$, if also the first assumption in \eqref{strana1} is satisfied (namely, 
$\mu_Y$- ess sup $|D_{H_{Y}}F|_{H_Y} <+\infty$, $\mu_Y$- ess sup $L_YF <+\infty$) then the trace operator is bounded from $W^{1,p}(\O, \mu)$ to $L^p($graph$\,F, \rho)$ for every $p>1$.

\subsection{Balls and ellipsoids in Hilbert spaces}

Let $X$ be a separable Hilbert space endowed with a nondegenerate centered Gaussian measure $\mu$, with covariance $Q$. By Trace$\,Q$ we mean as usual the sum of its eigenvalues, which of course has nothing to do with the trace operator. As mentioned in Section 2, we fix an orthonormal  basis $\{e_k:\;k\in \N\}$ of $X$  consisting of  eigenvectors of $Q$, $Qe_k= \lambda_ke_k$, and the corresponding orthonormal basis of $H= Q^{1/2}(X)$ is $\mathcal{V} = \{v_k:= \sqrt{\lambda_k}e_k:\;k\in \N\}$. For each $k$, the function $\hat{v}_k$ is just $\hat{v}_k(x) = x_k/\sqrt{\lambda_k} $, where $x_k =\langle x, e_k\rangle_X$.

For every $r>0$ the function $G(x) := \|x\|^2 - r^2$ satisfies Hypothesis \ref{HypG}. Indeed, it is smooth, $\O = B(0,r)$, $D_HG(x) =2Qx$ and $1/|D_HG|_H = 1/2\|Q^{1/2}x\|$ is easily seen to belong to   $L^p(X, \mu)$ for every $p$. 

Then, for each $p>1$ and  $\varphi\in W^{1,p}(B(0,r), \mu)$, \eqref{parti} holds, and it  reads as 
$$\int_{B(0,r)} D_k\varphi \,d\mu =  \frac{1}{\sqrt{\lambda_k} }  \int_{B(0,r)} x_k\varphi \,d\mu +   \int_{\|x\| = r} \frac{\sqrt{\lambda_k} x_k}{\|Q^{1/2}x\| }\varphi \,d\rho , \quad k\in \N .$$
Moreover  the trace operator is bounded from $W^{1,p}(B(0,r), \mu)$ to $L^{q}(\partial B(0,r), \rho)$ for every $1\leq q <p$. 
The question whether it is bounded from $W^{1,p}(B(0,r), \mu)$ to $L^{p}(\partial B(0,r), \rho)$ is not obvious, and it is related to the occurrence of a Hardy type inequality, as the next proposition shows. By ``Hardy type inequality" we mean
\begin{equation}
\label{Hardy0}\exists C>0:\;\;\int_{B(0,r)}\frac{|\varphi|^p}{\|Q^{1/2}x\|^p}\,d\mu \leq C\|\varphi\|_{W^{1,p}(B(0,r), \mu)}^p, \quad \varphi \in W^{1,p}(B(0,r), \mu).
\end{equation}
We need in fact a consequence of \eqref{Hardy0}, and precisely
\begin{equation}
\label{Hardy}
\exists C>0:\;\;\int_{B(0,r)}\frac{|\varphi|^p}{\|Q^{1/2}x\|}\,d\mu \leq C\|\varphi\|_{W^{1,p}(B(0,r), \mu)}^p, \quad \varphi \in W^{1,p}(B(0,r), \mu).
\end{equation}

\begin{Proposition}
If \eqref{Hardy} holds, then the trace operator is bounded from $W^{1,p}(B(0,r), \mu)$ to   $L^{p}(\partial B(0,r), \rho)$. Conversely, let $\lambda_{\max}$ be the maximum eigenvalue of $Q$. 
If the trace operator is bounded from $W^{1,p}(B(0,r), \mu)$ to   $L^{p}(\partial B(0,r), \rho)$ and $r^2 < \mbox{\rm Trace}\,Q   - \lambda_{\max}$, then 
\eqref{Hardy} holds. 
\end{Proposition}
\begin{proof}
Since  $LG(x) = 2($Trace$\,Q - \|x^2\|)$, for every Lipschitz continuous $\varphi$  formula \eqref{partitraccia2} reads as 
\begin{equation}
\label{sfera}
\begin{array}{lll}
\ds  \int_{\{\|x\|=r\}}  | \varphi|^p   \,d\rho  &= & \ds  p \int_{B(0,r)} |\varphi|^{p-2} \varphi \frac{\langle D_H\varphi, Qx \rangle_H }{\|Q^{1/2}x\|}  \,d\mu +  \int_{B(0,r)} \frac{ \mbox{\rm Trace}\,Q - \|x^2\|}{\|Q^{1/2}x\|} \,| \varphi|^p \,d\mu 
\\
\\
& -  &\ds  \int_{B(0,r)} \frac{\|Qx\|^2}{\|Q^{1/2}x\|^3} \,| \varphi|^p \,d\mu .
\end{array}
\end{equation}
As we already remarked, the first integral in the right hand side is harmless, since $ |Qx |_H= \|Q^{1/2}x\|_X$ so that 
$$\bigg|\int_{B(0,r)}  |\varphi|^{p-2} \varphi \frac{\langle D_H\varphi, Qx \rangle_H }{\|Q^{1/2}x\|}  \,d\mu\bigg| \leq 
\int_{B(0,r)}  |\varphi|^{p-1} |D_H\varphi|_H \,d\mu \leq  \|\varphi\|_{W^{1,p}(B(0,r))}^p .$$
Therefore, $ \int_{\{\|x\|=r\}}  | \varphi|^p   \,d\rho$ is bounded by  $\|\varphi\|_{W^{1,p}(B(0,r))}^p$ (up to a multiplication constant) iff there is $c$ such that 
$$I:=  \int_{B(0,r)} \bigg( \frac{ \mbox{\rm Trace}\,Q - \|x^2\|}{\|Q^{1/2}x\|} - \frac{\|Qx\|^2}{\|Q^{1/2}x\|^3} \bigg)\,| \varphi|^p \,d\mu \leq c\;\|\varphi\|_{W^{1,p}(B(0,r))}^p.$$
If \eqref{Hardy} holds,  then $I\leq C\,\mbox{\rm Trace}\,Q\|\varphi\|_{W^{1,p}(B(0,r))}^p$, and the first statement follows. 

Concerning the converse, since $\|Qx\|^2\leq \|Q^{1/2}\|^2 \|Q^{1/2}x\|^2\leq \lambda_{\max}\|Q^{1/2}x\|^2$, 
then 
$$ ( \mbox{\rm Trace}\,Q - r^2 - \lambda_{\max}) \int_{B(0,r)}\frac{|\varphi|^p}{ \|Q^{1/2}x\|} \,d\mu \leq I.$$
Therefore, if $r$ is small enough (namely, $r^2 < \mbox{\rm Trace}\,Q   - \lambda_{\max}$) and the trace is bounded from $W^{1,p}(B(0,r),  \mu)$ to   $L^{p}(\partial B(0,r), \rho)$, then  \eqref{Hardy} holds for every Lipschitz continuous $\varphi$ and hence for every $\varphi \in W^{1,p}(B(0,r), \mu)$. \end{proof}

However, the occurrence of \eqref{Hardy} is an open problem, related to other open problems in the theory of Sobolev spaces in infinite dimensions. For instance, if a bounded  extension operator ${\mathcal E }$ from $W^{1,p}(B(0,r), \mu)$ to $W^{1,p}(X, \mu)$ existed, then  \eqref{Hardy}  would be a consequence of the Hardy type inequality
\begin{equation}
\label{HardyX}
\exists C>0:\;\;\int_{X}\frac{|\varphi|^p}{\|Q^{1/2}x\|^p}\,d\mu \leq C\|\varphi\|_{W^{1,p}(X, \mu)}^p, \quad \varphi \in W^{1,p}(X, \mu), 
\end{equation}
that is easily seen to hold, under suitable assumptions on $Q$. But existence of a bounded extension operator is still an open problem. 

\vspace{3mm}

The above results may be extended without important modifications to balls centered at $x_0\neq 0$, and to ellipsoids defined by $E = \{ x\in X:\; \sum_{k=1}^{\infty}\alpha_k x_k^2 <r^2\}$ for some bounded nonnegative  sequence $(\alpha_k)$, not eventually vanishing (if $\alpha_k\neq 0$ for finitely many $k$, $E$ is a cylindrical set). Instead, the case where $(\alpha_k)$ is unbounded needs some more attention.

\begin{Lemma}
\label{serie}
Let $(\alpha_k)$ be any  sequence of nonnegative numbers, not eventually vanishing. Then the function
$$G(x) = \sum_{k=1}^{\infty} \alpha_k x_k^2$$
belongs to $L^1(X, \mu)$ iff 
\begin{equation}
\label{ns}
\sum_{k=1}^{\infty} \lambda_k\alpha_k <\infty , 
\end{equation}
and in this case $G \in W^{2,p}(X, \mu)$, it is $C_{2,p}$-quasicontinuous, and $1/|D_HG|_H\in L^p(X, \mu)$ for every $p>1$. Moreover, setting $E_r=\{ x\in X:\; \sum_{k=1}^{\infty}\alpha_k x_k^2 \leq r^2\}$,  $\mu(E_r)>0$ for every $r>0$. 
\end{Lemma}
\begin{proof}
The fact that $G\in L^1(X, \mu)$ iff \eqref{ns} holds follows immediately from the equality $\int_X x_k^2\,d\mu = \lambda_k$ for every $k\in \N$. In this case $G\in L^p (X, \mu)$ for every $p$, since for every $n\in \N$
we have $\int_X (G(x)^n)d\mu \leq C_n (\sum_{k=1}^{\infty}\lambda_k\alpha_k)^n$ for some $C_n>0$. Moreover, for every $k\in \N$ we have $D_kG(x)= 2\sqrt{\lambda_k} \alpha_kx_k$ and $D_{hk}G(x)= \lambda_k\alpha_k $ if $h=k$, $D_{hk}G(x)= 0$ if $h\neq k$, therefore 
$G \in W^{2,p}(X, \mu)$ for every $p$. 

\vspace{1mm}
Let us prove that $1/|D_HG(x)|_H \in L^p(X, \mu)$ for every $p>1$. 
Let $(\alpha_{k_n})$ be any subsequence  of  $(\alpha_k)$ assuming strictly positive values. 
Fix $p>1$ and let $n\in \N$, $n>p$. Then
$$|D_HG(x)|_H^p = 2^p\bigg(\sum_{k=1}^{\infty}\lambda_k\alpha_k^2 x_k^2\bigg)^{p/2}\geq c_{n,p}
\bigg(\sum_{k=1}^{n} x_{k_j}^2\bigg)^{p/2}, $$
with $c_{n,p}:= 2^p(\min\{ \lambda_{k_j}\alpha_{k_j}^2:\; j=1, \ldots , n\})^{p/2}$. Therefore, 
$$\int_X \frac{1}{|D_HG(x)|_H^p}\,d\mu \leq  \frac{1}{c_{n,p}}\frac{1}{(2\pi)^{n/2}}  \frac{1}{(\Pi_{j=1}^{n} \lambda_{k_j})^{1/2}}\int_{\R^n} \bigg(\sum_{k=1}^{n} x_{k_j}^2\bigg)^{-p/2}e^{-\sum_{j=1}^nx_{k_j}^2/2\lambda_{k_j}}dx_{k_1}\cdots dx_{k_n}$$
which is finite since $n>p$.

\vspace{1mm}
If the sequence $(\alpha_k)$ is bounded, then $G$ is continuous. If $(\alpha_k)$ is unbounded, to prove that $G$ is $C_{2,p}$-quasicontinuous it is sufficient to follow the proof of \cite[Thm. 5.9.6]{Boga}, taking 
$F_n = \sum_{k=1}^n \alpha_k x_k^2$, $f_n = F_n- L_pF_n$, $T= (I-L_p)^{-1}$, where $L_p$ denotes as usual the realization of the Ornstein--Uhlenbeck operator in $L^p(X, \mu)$, for any $p>1$. 

\vspace{1mm}
Let us prove that $\mu(E_r)>0$ for every $r>0$. 
As before, let  $(\alpha_{k_n})$ be any subsequence  of  $(\alpha_k)$ assuming strictly positive values. 
Let $\widetilde{X} =$ span$\,\{e_{k_n}:\; n\in \N\}$ be endowed with the scalar product of $X$ and with the Gaussian measure $\widetilde{\mu} :=N_{0,\widetilde{Q}}$, where $\widetilde{Q}$ is the diagonal operator defined by
$\widetilde{Q}e_{k_n} := \lambda_{k_n} \alpha_{k_n} e_{k_n}$, that has finite trace by \eqref{ns}.  

Setting $E_{r,n}:= \{ x\in X:\; \sum_{k=1}^{k_n} \alpha_kx_k^2 \leq r^2\}$, we have $E_{r,n+1}\subset E_{r,n}$ for every $n$ and $E_{r} = \cap_{n\in \N} E_{r,n}$, so that $\mu(E_r) = \lim_{n\to \infty} \mu( E_{r,n})$. Similarly, denoting by $B(0,r)$ the ball centered at $0$ with radius $r$ in $\widetilde{X}$ and
setting $B_{r,n}:= \{ x\in \widetilde{X}:\; \sum_{j=1} ^n x_{k_j}^2 \leq r^2\}$,
we have $ \widetilde{\mu}(B(0,r)) = \lim_{n\to \infty} \widetilde{\mu}( B_{r,n})$. 
On the other hand, for every $n\in \N$ we have
$$\begin{array}{lll}
\mu (E_{r,n}) & = & \ds \frac{1}{(2\pi)^{k_n/2}} \frac{1}{(\Pi_{k=1}^{k_n} \lambda_k)^{1/2}}
\int_{\{x=(x_1, \ldots, x_{k_n})\in \R^{k_n}:\; \sum_{k=1}^{k_n} \alpha_k x_k^2 \leq r^2 \}} e^{-\sum_{k=1}^{k_n}x_k^2/2\lambda_k}dx_{1}\ldots dx_{k_n}
\\
\\
& = & \ds \frac{1}{(2\pi)^{n/2}} \frac{1}{(\Pi_{j=1}^n \lambda_{k_j} )^{1/2}}
\int_{\{ x=(x_{k_1}, \ldots, x_{k_n}) \in \R^n: \; \sum_{j=1}^{n} \alpha_{k_j} x_{k_j}^2 \leq r^2 \}} 
 e^{-\sum_{j=1}^nx_{k_j}^2/2\lambda_{k_j}}dx_{k_1}\ldots dx_{k_n}
\end{array}$$
and changing variables, $y_j = \sqrt{\alpha_{k_j}}x_{k_j}$ for $j=1, \ldots, n$, we get 
$$\mu (E_{r,n})   = \frac{1}{(2\pi)^{n/2}} \frac{1}{(\Pi_{j=1}^n \lambda_{k_j}\alpha_{k_j})^{1/2}}
\int_{\{ y\in \R^n: \; \sum_{j=1}^{n} y_{j}^2 \leq r^2 \}} 
 e^{-\sum_{j=1}^ny_{j}^2/2\lambda_{k_j}\alpha_{k_j}}dy_{1}\cdots dy_{n}
 = \widetilde{\mu}( B_{r,n}).$$
Letting $n\to \infty$ we obtain $\mu(E_r) =  \widetilde{\mu}(B(0,r)) >0$, since $\widetilde{\mu}$ is non degenerate in $ \widetilde{X}$. 
\end{proof}

By Lemma \ref{serie}, if condition \eqref{ns} holds and $\alpha_k > 0$ for infinitely many $k$, the function $x\mapsto G(x) - r^2$ satisfies Hypothesis \ref{HypG}. Let us give a significant example. 

\begin{Example}
{\em Let $X=L^2((0, 1), d\xi)$ and let $-A$ be the realization of the second order derivative with Dirichlet boundary condition, i.e. $D(A)=W^{2,2}((0, 1), d\xi) \cap W^{1,2}_0((0, 1), d\xi)$, $Ax = -x''$. As orthonormal basis of $X$ we choose the set of the eigenfunctions of $A$, $e_k(\xi) := \sqrt{2}\sin(k \pi \xi)$, $k\in \N$, with eigenvalues $(\pi k)^2$. For every $\beta >0$ we have 
$$D(A^{\beta}) = \{ x\in X:\; \sum_{k=1}^{\infty} k^{4\beta}x_k^2 <\infty\}, \quad A^{\beta}x =  \sum_{k=1}^{\infty} (\pi k)^{2\beta}x_ke_k. $$
The open ball  centered at $0$ with radius $r$ in $D(A^{\beta})$ is denoted by $B_{\beta}(0,r)$. Moreover we set
$$G(x) =  \|A^{\beta}x\|^2 - r^2,$$
so that $G^{-1}(-\infty, 0) = B_{\beta}(0,r)$. 

(i) We consider  the Gaussian measure $\mu $ in $X$ with mean $0$ and covariance $Q :=  \frac12\,A^{-1}$. Then, the eigenvalues of $Q$ are $\lambda_k = 1/(2k^2 \pi^2)$. 
Choosing $\alpha_k = (\pi k)^{4\beta}$ with $\beta < 1/4$, condition \eqref{ns} is satisfied. By Lemma \ref{serie},  the function $G$ satisfies Hypothesis \ref{HypG}. Moreover, $|D_HG(x)|_H = 2 \|A^{-1+2\beta}x\|$. The integration   formula \eqref{partiW1p} on $\O = B_{\beta}(0,r)$ reads as
$$\int_{B_{\beta}(0,r)} D_k\varphi \,d\mu =  \sqrt{2}\pi k \int_{B_{\beta}(0,r)} x_k\varphi \,d\mu 
+ \frac{1}{(\pi k)^{1-2\beta}} \int_{\partial B_{\beta}(0,r)} \frac{x_k}{\|A^{-1+2\beta}x\|}\,\Tr \varphi\, d\rho,   $$
for every  $\varphi\in W^{1,p}(B_{\beta}(0,r), \mu) $ and $k\in \N$. Here $\partial B_{\beta}(0,r) = G^{-1}(0)$ is the boundary of $B_{\beta}(0,r)$ in $D(A^{\beta})$.

(ii) Next, we choose as $\mu$ the Gaussian measure with mean $0$ and covariance $Q :=  \frac12\,A^{-2}$. The eigenvalues of $Q$ are in this case $\lambda_k = 1/(2k^4 \pi^4)$. Choosing again $\alpha_k = (\pi k)^{4\beta}$ with $\beta < 3/4$, condition \eqref{ns} is satisfied. So,  the function $G $ satisfies Hypothesis \ref{HypG},  and $|D_HG(x)|_H = 2 \|A^{-2+2\beta}x\|$. The integration   formula \eqref{partiW1p} reads now as
$$\int_{B_{\beta}(0,r)} D_k\varphi \,d\mu =  \sqrt{2}\pi^2 k^2 \int_{B_{\beta}(0,r)} x_k\varphi \,d\mu 
+ \frac{1}{(\pi k)^{2-2\beta}} \int_{\partial B_{\beta}(0,r)} \frac{x_k}{\|A^{-2+2\beta}x\|}\,\Tr \varphi\, d\rho, $$
for every  $\varphi\in W^{1,p}(B_{\beta}(0,r), \mu) $ and $k\in \N$. }
\end{Example}

\section{Acknowledgements}

It is a pleasure to thank  L. Tubaro for hours of discussions on the subject of this paper  and on related items.
 


\begin{thebibliography}{99}

\bibitem{AM}
H. Airault,  P. Malliavin, \textit{Int\'egration g\'eom\'etrique sur l'espace de Wiener},  Bull. Sci. Math. {\bf 112} (1988),  3--52. 

\bibitem{AVN} J. Assaad, J. Van Neerven, \textit{$L^2$-theory for non non-symmetric Ornstein--Uhlenbeck semigroups on domains}, to appear in J. Evol. Equ. Published online, DOI 10.1007/s00028-012-0171-1. 

\bibitem{Diego0}
L. Ambrosio,  M. Miranda Jr., S. Maniglia,  D. Pallara, 
\textit{BV functions in abstract Wiener spaces}, J. Funct. Anal. {\bf 258} (2010), 785--813. 

\bibitem{Diego}
L. Ambrosio,  M. Miranda Jr.,  D. Pallara, 
\textit{Sets of finite perimeter in Wiener spaces, perimeter measure and boundary rectifiability}, Discr. Cont. Dynam. Systems {\bf 28} (2010), 591--608. 

\bibitem{BDPT1} V.  Barbu, G. Da Prato, L. Tubaro,  \textit{ Kolmogorov equation associated to the stochastic reflection problem
on a smooth convex set of a Hilbert space}, Ann. Probab. {\bf 37} (2009), 1427--1458. 

\bibitem{BDPT2} V.  Barbu, G. Da Prato, L. Tubaro,   \textit{Kolmogorov equation associated to the stochastic reflection problem on a smooth convex set of a Hilbert space II},  Ann. Inst. H. Poincar\'e Probab. Stat. {\bf 47} (2011), 699--724.
   
\bibitem{Boga}  V.I. Bogachev, \textit{Gaussian Measures}, American Mathematical Society, Providence, 1998.

\bibitem{DP} G. Da Prato, \textit{An introduction to infinite dimensional analysis}, Springer-Verlag, Berlin 2006. 

\bibitem{DPL} G. Da Prato, A. Lunardi, \textit{On the Dirichlet semigroup for Ornstein--Uhlenbeck 
operators in subsets of Hilbert spaces}, J. Funct. Anal. {\bf 259} (2010), 2642--2672. 

\bibitem{DPL2} G. Da Prato, A. Lunardi, \textit{Sobolev regularity for a class of second order elliptic PDE's in infinite dimension}, to appear in J. Maths. Pures Appl. Published online, DOI 10.1016/j.matpur.2012.10.008. 

\bibitem{DPZ} G. Da Prato, J. Zabczyk, \textit{Second Order Partial Differential Equations in Hilbert spaces},
London  Math. Soc. Lect. Notes no. 293, Cambridge University Press, Cambridge (2002).

\bibitem{EG}   L.C. Evans, R.F. Gariepy, \textit{Measure Theory and Fine Properties of Functions},  CRC Press, Boca Raton (1999). 

\bibitem{Feyel}   D. Feyel,  \textit{Hausdorff-Gauss Measures}, in: Stochastic Analysis and Related Topics, VII. Kusadasi 1998, Progr. in Probab. 98, Birkh\"auser, Boston 2001, 59--76. 

\bibitem{FP1}   D. Feyel,  A. de La Pradelle, \textit{Capacit\`es Gaussiennes}, Ann. Inst. Fourier {\bf 41} (1991), 49--76. 
 
\bibitem{FP}   D. Feyel,  A. de La Pradelle, \textit{ Hausdorff measures on the Wiener space}, Pot. Analysis {\bf 1} (1992), 177--189. 
 
 \bibitem{FH} M. Fukushima, M. Hino, \textit{On the space of BV functions and a related stochastic calculus in infinite dimensions},
J. Funct. Anal. {\bf 183} (2001), 245-268.
 
\bibitem{Hino} M. Hino, \textit{Sets of finite perimeter and the Hausdorff--Gauss measure on the Wiener space}, J. Funct. Anal. {\bf 258} (2010), 1656--1681.



\bibitem{Tr}
H. Triebel,  \textit{Interpolation Theory, Function Spaces, 
Differential Operators}, North-Holland, Amsterdam, 1978.

\end{thebibliography}
\end{document}